\documentclass[10pt]{article}
\title{ Optimality Conditions and Duality for Multiobjective Fractional 
 Bilevel Optimization Problems} 
\author{Felipe Lara\thanks{Instituto de Alta Investigaci\'on (IAI), Universidad 
de Tarapac\'a, Arica, Chile. E-mail: felipelaraobreque@gmail.com;
flarao@academicos.uta.cl. Web: felipelara.cl, Orcid-ID: 0000-0002-9965-0921} 
\and
Rishabh Pandey\thanks{Department of Mathematics, National Institute of 
Technology, Chaltlang, Aizawl, 796012, Mizoram, India. E-mail:
 pandeyrishabh512@gmail.com}
\and
Vinay Singh\thanks{Department of Mathematics, National Institute of 
Technology, Chaltlang, Aizawl, 796012, Mizoram, India. E-mail: vinay.math@nitmz.ac.in. Orcid-ID: 0000-0003-3269-8776}}
\usepackage{amssymb}
\usepackage{graphics}
\usepackage{amsmath}
\usepackage{multicol}
\usepackage{amsthm, amssymb, amsfonts}
\usepackage{enumitem}
\usepackage{tabularx}
\usepackage{amsmath}
\usepackage{amsthm}
\usepackage{amssymb}
\usepackage{multicol}
\usepackage[active]{srcltx}
\usepackage{array}
\usepackage[dvipsnames,svgnames]{xcolor}

\newtheorem{theorem}{Theorem}[section]

\newtheorem{definition}[theorem]{Definition}
\newtheorem{example}[theorem]{Example}

\newtheorem{lemma}[theorem]{Lemma}

\newtheorem{remark}[theorem]{Remark}

\begin{document}

\date{\today}

\maketitle

\begin{abstract}
 \noindent This paper studies a multiobjective bilevel optimization problem where each objective is a fractional function. By reformulating the problem into a single-level one, we establish refined necessary and sufficient optimality conditions. These results are derived using ${\partial}_D$-nonsmooth Abadie-type constraint qualifications and generalized convexity concepts (quasiconvexity and pseudoconvexity) based on directional convexificators. We also prove weak and strong duality theorems for a Mond-Weir dual problem formulated with directional convexificators. Finally, several examples are provided to illustrate the advantages of our approach.

\medskip


\medskip

\end{abstract}

\section{Introduction}

Bilevel programming is an interesting research field in optimization and variational analysis, which has attracted the attention of a huge number of researchers from applied mathematics in virtue of its applications in different areas as economics, engineering, finance, chemistry and computer science among others (see \cite{dempe2002foundations,dempe2015bilevel} and references therein).

The bilevel programming problem involves optimizing a lower level pro\-blem to assess each upper level solution, leading to computational  challenges. Theo\-re\-tically, an upper level solution is deemed valid or feasible only if the corresponding lower level variables represent the true global optimum of the  lower level problem. The study of these problems involves the analysis of a pair of interconnected optimization problems in which the set of feasible solutions for the upper level problem is implicitly determined by the set of solutions for the lower level problem.

Bilevel optimization has been extensively studied in the literature \cite{bard1984optimality,chuong2020optimality,gadhi2012necessary,gadhi2018equivalent,gadhi2019sufficient,gadhi2023multiobjective,IL-1,IL-3,lafhim2018necessary,Lara-4,ye2006constraint}. For instance, Bard \cite{bard1984optimality} investigated linear bilevel programming problems and developed first-order necessary optimality conditions. In a different approach, Lafhim and coauthors \cite{lafhim2018necessary} employed scalarization techniques and a generalized Abadie constraint qualification to derive KKT-type necessary optimality conditions for multiobjective bilevel problems. Furthermore, Gadhi \cite{gadhi2018equivalent} established necessary optimality conditions for a semivectorial bilevel problem under a special scalarization function, while in a later work \cite{gadhi2023multiobjective}, the same author developed necessary optimality conditions for multiobjective bilevel programming problems in terms of tangential subdifferentials.

Let $\mathcal{F}_k, \mathcal{G}_k: \mathbb{R}^{n_1} \times \mathbb{R}^{n_2} \rightarrow \mathbb{R}$ be real-valued functions for every $k \in \{1,\ldots,n\}$ and $\mathcal{H}_j: \mathbb{R}^{n_1} \times \mathbb{R}^{n_2} \rightarrow \mathbb{R}$ be real-valued functions for every $j \in J := {\{1,\ldots,p\}}$. Then, the multiobjective fractional bilevel optimization problem given by:
\begin{equation}\label{problem:a} \text{(A)}:\quad 
\begin{cases}
\underset{\qquad \ x, y}{\mathop{\mathbb{R}^{n}_+ - \min }}\ \frac{{\mathcal{F}} (x,y)}{{\mathcal{G}}(x,y)}= \left(\frac{{\mathcal{F}_1}(x,y)}{{\mathcal{G}_1} (x,y)}, \dots, \frac{{\mathcal{F}_n} (x,y)}{ {\mathcal{G}_n} (x,y)} \right) \\
s. \ t. ~~ {\mathcal{H}_j} (x,y) \leq 0, ~ \forall ~ j \in J \\
\qquad ~ y \in \mathfrak{B} (x)
\end{cases}
\end{equation}
where, for every $x \in \mathbb{R}^{n_1}$, \textbf{$\mathfrak{B} (x)$} is the solution set of the lower level problem
\begin{equation}\label{problem:ax}
(\text{A}_{x}):\quad 
\begin{cases}
 \underset{ y}{\mathop{ \min }} \ \ f (x, y)\\
 s.\ t. \quad \phi_s(x,y) \leq 0,\ \   \forall\ s \in S,
\end{cases}
\end{equation}
in which $f: \mathbb{R}^{n_1} \times \mathbb{R}^{n_2} \rightarrow 
\mathbb{R}$ and $\phi_s: \mathbb{R}^{n_1} \times \mathbb{R}^{n_2} \rightarrow \mathbb{R}$ for every $s \in S := { \{1,\ldots,q\}}$ are real-valued functions as well. Furthermore, we will assume that $\mathcal{F}_k (x, y) \geq 0$ and $\mathcal{G}_k (x, y) > 0$ for all $(x, y) \in \mathcal{M}$ and all $k \in \{1,\ldots,n\}$, where $\mathcal{M} := \{(x, y) \in \mathbb{R}^{n_1} \times \mathbb{R}^{n_2}: \mathcal{H}_j (x, y) \leq 0, ~ \forall ~ j \in J, ~ \forall ~ y \in \mathfrak{B} (x)\}$. 

Fractional bilevel programming has gained considerable interest because of its ability to capture hierarchical decision structures involving ratio-based performance measures. In this context, Kohli \cite{kohli2023optimality} established optimality conditions for multiobjective fractional bilevel problems using convexificators, providing an effective tool for handling nonsmooth fractional objectives. Saini and Kailey \cite{saini2024duality} further contributed to the field by deriving duality results for multiobjective fractional bilevel programming and demonstrating the usefulness of convexificator-based approaches. More recently, Pandey et al. \cite{pandey2025approximate} developed $\varepsilon$-convexificator techniques to obtain both necessary and sufficient optimality conditions for multiobjective fractional bilevel problems. 

As usual in optimization, convexity and differentiability are desired properties for studying mathematical models, since local properties are also global, and using differentiability, we could obtain optimality conditions which are easy to check. However, these properties often fail for nonconvex and/or nonsmooth functions, the usual tools from convex analysis fail for providing adequate information as optimality and, for these reasons, the researchers have been developed several notions of generalized convexity and generalized differentiability, in order to obtain adequate information from the data even in these hard frameworks.

This paper focuses on fractional objective functions, which are inherently nonconvex. For instance, even when differentiable functions when the numerator is convex and the denominator is affine, the fractional function is not even convex, but it is pseudoconvex (see \cite[Theorem 3.2.10]{CM-Book}). Fractional functions have been studied extensively in the literature due to their applications in various fields, such as in economics for maximizing return/risk or minimizing cost/time (see \cite{CM-Book,SCH,STA} and the references therein).

Furthermore, in this paper we analyze our main problem without differentiability assumptions and, to that end, we use the convexificator notion \cite{jeyakumar1999nonsmooth} 
for studying the multiobjective fractional bilevel optimization problem $(A)$. 
Recall that the convexificator is an extension of the sub\-di\-ffe\-rential notion. For instance, for locally Lipschitz functions, most well-known subdifferentials 
are convexificators, and these subdifferentials may encompass the convex hull 
of a convexificator by \cite{jeyakumar1999nonsmooth} (see also \cite{KS}). 
However, an important drawback is that convexificators may be unbounded 
for discontinuous functions. For this reason, the authors in
\cite{dempe2015necessary} introduced the concept of directional convexificators
motivated for lower semicontinuous (lsc henceforth) functions. This notion takes into account the directions where 
the function exhibits continuity and it has been proved to be useful for providing
optimality conditions in optimistic bilevel programming problems. Furthermore, 
the authors in \cite{gadhi2021optimality} provided optimality conditions for a 
set-valued optimization problem using support functions of set-valued mappings 
while in \cite{gadhi2022optimality} optimality conditions for mpecs in terms of
directional upper convexificators were established. We refer to
\cite{el2022applying,gadhi2022variational,lafhim2023optimality} for recent 
advances on directional convexificators.



\begin{table}[h!]
\centering
\begin{tabular}{|p{2.1cm}|p{3.8cm}|p{2.4cm}|p{2.6cm}|}
\hline
\textbf{References} & \textbf{Bilevel Model} & \textbf{Tool} & \textbf{Results} \\ \hline

Kohli, B. \cite{kohli2023optimality} &
Multiobjective fractional bilevel programming problem &
Convexificator &
Necessary and sufficient condition \\ \hline

Saini, S., Kailey, N. \cite{saini2024duality} &
Multiobjective fractional bilevel programming problem &
Convexificator &
Mond–Weir dual \\ \hline

Pandey et al. \cite{pandey2025approximate} &
Multiobjective fractional bilevel programming problem &
$\varepsilon$-convexificator &
Necessary and sufficient optimality \\ \hline

In this paper &
Multiobjective fractional bilevel programming problem &
Directional convexificator &
Necessary, sufficient and duality \\ \hline

\end{tabular}
\caption{Comparison of different works on multiobjective fractional bilevel programming}
\label{tab:comparison}
\end{table}

By using a novel function $\Psi$ from \cite{gadhi2012necessary}, we reformulate 
the original problem \text{(A)} into a simple and single-level problem \text{$(A^*)$}.
Under a suitable constraint qualification, we derive the necessary optimality conditions based on directional convexificators. Furthermore, we leverage a generalized convexity concept from \cite{el2022applying} to obtain first order sufficient optimality conditions. Moreover, we formulate a Mond-Weir-type dual problem for \text{(A)} and we establish both weak and strong duality results. Finally, we illustrate our findings through concrete examples.

The structure of the paper is as follows: In Section \ref{sec:02}, we present preliminaries and basic definitions. In Section \ref{sec:03}, we develop necessary
optimality conditions while sufficient optimality conditions under generalized convexity notions are given in Section 4. Finally, in Section 5, we focus on Mond-Weir duality type results.

\section{Preliminaries}\label{sec:02}

We denote zero in $\mathbb{R}^{n}$ by $0_n := (0,\ldots, 0)$ and a sequence of positive real numbers $(t_k)$ approaching zero by $t_k \downarrow  0$. For a nonempty subset $S \subseteq \mathbb{R}^{n}$, we denote the distance function by $d(y,  S) := \inf_{a \in S} \, \lVert y - a\rVert$, while the notations of $bd\,S$, $int\,S$, $co\,S$, $cl\,S$, $pos\,S$ and $cone\,S$ means topological boundary, topological interior, convex hull, closure, convex cone (including the origin) and cone of $S$, respectively, where $cone\,S := \cup_{t \geq 0} \, t \, S$. The negative polar cone of $S$ is defined by
$$S^\circ := \{u \in \mathbb{R}^n : \langle u, x \rangle \leq 0, ~ \forall\ x \in S\}.$$
Let $S_1, S_2 \subseteq \mathbb{R}^{n}$ be two closed convex cones. Then $(S_1\cap S_2)^\circ = cl ~({S_1^{\circ} + S_2^{\circ}})$.

The contingent cone and the cone of weak feasible directions to the set $S$ at $\check{x} \in  cl~S$  are given, respectively, by
\begin{align*}
 & ~ T(S, \check{x}) := \big\{u \in \mathbb{R}^{n}: \, \exists\ t_k \downarrow  0,\ \exists\ u_k \rightarrow u \ \ \text{such that}\ \check{x} + 
 t_k u_k \in S,\ \forall\ k \geq 0 \big\}, \\
 & W (S, \check{x}) := \big\{u \in \mathbb{R}^{n}: \, \exists\ t_k 
 \downarrow  0 \ \ \text{such that}\ \check{x} +t_k u \in S,\ \forall\ k \geq 0 \big\}.
\end{align*}
Note that, for all $\check{x} \in  cl\,S$, we have
\begin{equation}\label{eq01}
  W (S, \check{x}) \subseteq T (S, \check{x}). 
\end{equation}
 The Normal cone $N_S(\check{x})$ to $S$ at $\check{x} \in  cl~S$ is given by $$N_S(\check{x}) = T (S, \check{x})^{\circ} = \big\{ u  \in \mathbb{R}^{n}: \langle  u, \rho \rangle \leq 0,~ \forall~  \rho \in T (S, \check{x}) \big\}.$$

 If $S$ is convex, then $T (S, \check{x})$ is convex and 

$$N_S(\check{x})= \{u\in \mathbb{R}^{n}: \langle u,~ x-\check{x} \rangle \leq 0, ~ \forall~ x \in S \}.$$

 It is said that a nonempty set $S \subseteq \mathbb{R}^{n}$ is locally 
 star-shaped at $\check{x} \in S,$ if for every $x \in S$, there exists a scalar
 $a (\check{x}, x) \in (0,1]$ such that 
 $$\check{x}+ \lambda (x-\check{x}) \in S, ~ \forall~ \lambda \in (0, a (\check{x},~x)).$$
 If $a (\check{x},~x) = 1$, then $S$ is star-shaped at $\check{x}$.

Every open or convex set is locally star-shaped at each of its points. For example, any cone is locally star-shaped at its vertex, even if it is not locally convex. According to \cite{kaur1983theoretical}, if a closed set $S$ is locally star-shaped at every point $\check{x} \in S$, then $S$ must be convex. Furthermore, the finite union or intersection of sets that are locally star-shaped at a common point $\check{x}$ is also locally star-shaped at $\check{x}$. Finally, if $S$ is locally star-shaped at $\check{x}$, then the tangent cone satisfies $T(S, \check{x}) = cl,W (S, \check{x})$.

In order to express the bilevel optimization problem $(A)$ as a single-level mathematical programming problem, we define the function $\Delta_S: \mathbb{R}^{n}\rightarrow \mathbb{R}_+$ by (see \cite{hiriart1979tangent} and also \cite{dempe2020optimality,ehrgott2005multicriteria})
\begin{equation*}
\Delta_S(x) = \left\{
\begin{array}{ccl}
 -d(x, \mathbb{R}^{n}\backslash S), & {\rm if} & x \in S, \\
 d(x,  S), & {\rm if} & x \in \mathbb{R}^{n}\backslash S.
\end{array}
\right.
\end{equation*}

\begin{lemma} {\rm (\cite{hiriart1979tangent,lafhim2018necessary})} \label{prp1}
 Let $S \subseteq \mathbb{R}^{n}\ (S \neq \mathbb{R}^{n})$ be a closed and 
 convex cone with ${\rm int}\,S \neq \emptyset$. Then $\Delta_S(x)$ is convex, 
 positively homogeneous, $1$-Lipschitzian and decreasing on $\mathbb{R}^{n}$ 
 with respect to the order introduced by $S$ and $0 \notin \partial\Delta_S(y)$.
 Furthermore, 
 $$ int\,(S) = \{ y \in \mathbb{R}^n :\Delta_S(y) <0 \},$$ 
 $$ (\mathbb{R}^{n}\backslash S) = \{ y \in \mathbb{R}^n :\Delta_S(y) >0 \},$$ 
 $$ bd\,(S) = \{ y \in \mathbb{R}^n :\Delta_S(y) = 0 \}.$$
\end{lemma}

Now, let us recall the usual solution notions for the multiobjective optimization problem in our setting.

A pair $(\check{x},\check{y}) \in \mathcal{M}$ is said to be a:
 \begin{itemize}
  \item[$(i)$] weak Pareto solution of $($\text{A}$),$ if
   \begin{equation*}
    \frac{\mathcal{F}_k(x,y)}{\mathcal{G}_k(x,y)} - \frac{\mathcal{F}_k 
    ( \check{x}, \check{y} )}{\mathcal{G}_k( \check{x}, \check{y} )}\ \notin 
    - int\ {\mathbb{R}^n_+}, ~ \forall ~ (x,y) \in \mathcal{M},
   \end{equation*}
that is, there is no $(x^0, y^0) \in  \mathcal{M} $ such that:
  \begin{equation*}
   \frac{\mathcal{F}_k(x^0, y^0)}{\mathcal{G}_k(x^0, y^0)} < 
   \frac{\mathcal{F}_k( \check{x}, \check{y} )}{\mathcal{G}_k 
   (\check{x}, \check{y} )}, ~ \forall ~ k =1,\ldots,n,
  \end{equation*}

  \item[$(ii)$] local weak Pareto solution of $($\text{A}$),$ if there exist two 
  neighborhoods $U^0$ of $\check{x}$ and $V^0$ of $\check{y}$ such that:
 \begin{equation*}
  \frac{\mathcal{F}_k (x, y)}{\mathcal{G}_k (x,y)} - \frac{\mathcal{F}_k (\check{x}, \check{y} )}{\mathcal{G}_k (\check{x}, \check{y} )} \ \notin - int\ {\mathbb{R}^n_+}, ~ \forall \ (x,y) \in \mathcal{M}\cap (U^0 \times V^0),
 \end{equation*}
that is, there is no $(x^0, y^0) \in  \mathcal{M} \cap (U^0 \times V^0)$ such 
that:
\begin{equation*}
 \frac{\mathcal{F}_k (x^0, y^0)}{\mathcal{G}_k (x^0, y^0)} < \frac{\mathcal{F}_k (\check{x}, \check{y} )}{\mathcal{G}_k (\check{x}, \check{y} )}, ~ \forall ~ k =1,\ldots,n,
\end{equation*}
where $\mathcal{M} = \{(x, y) \in \mathbb{R}^{n_1} \times \mathbb{R}^{n_2}:
\mathcal{H}_j (x, y) \leq 0, ~ \forall ~ j \in J, ~ \forall ~ y \in \mathfrak{B} (x)\}$.
\end{itemize}

Let $h: \mathbb{R}^{n} \rightarrow \mathbb{R} \cup \{+\infty\}$ be a function and $x \in \mathbb{R}^{n}$ be such that $h (x) \in \mathbb{R}$. Then the upper and lower Dini derivatives of $h$ at $x$ in direction $d \in \mathbb{R}^{n}$ are defined by
$$h^- (x;d) = \underset{t \searrow 0}{\mathop{\lim \inf}} \frac{h (x+td) - h (x)}{t},~~~~ h^ + (x;d) = \underset{t \searrow 0}{\mathop{\lim \sup}} \frac{h (x+td) - h (x)}{t},$$
respectively. If $h: \mathbb{R}^{n} \rightarrow {\mathbb{R}}$ is 
locally Lipschitz, then both of them exist.

 A vector $d \in \mathbb{R}^{n}$ is said to be a continuity direction of $h$ at $x \in \mathbb{R}^{n}$ if, for any real sequences $\{p_k\} \subset \mathbb{R}$ with $\{p_k\} \searrow 0$, we have  
 $$ \lim_{k \rightarrow \infty} h (x+p_kd)= h (x).$$
 We denote by $D(x)$ to the set of all continuity directions of $h$ at $x$.

Clearly, $D(x)$ is a nonempty cone (always contains $0_n$), but it is not nece\-ssa\-ri\-ly neither closed nor convex. When there is no potential confusion, we will denote $D(x)$ as $D$. Furthermore, note that the Frechet normal cone to $D$ at $\bar{d}=0_n$ is given by $N_D(0_n)=T(D, 0_n)^\circ$. 


 Let $h: \mathbb{R}^{n} \rightarrow \overline{\mathbb{R}}$ be a
 function. Then it is said that:
 \begin{itemize}
  \item[$(i)$] $h$ admit a directional upper convexificator (DUCF)
  ${\partial}^{u}_D h (x)$ at $x$ if the set ${\partial}^{u}_D h (x) \subseteq \mathbb{R}^{n}$ is closed and satisfies
  \begin{equation}
   h^- (x,d) \leq \underset{x^* \in {\partial}^{u}_D h (x)}{\mathop{\sup}}  \langle x^*, d\rangle, ~ \forall ~ d \in D.
  \end{equation}

  \item[$(ii)$] $h$ admit a directional upper semi-regular convexificator (DUSRCF) ${\partial}^{us}_D h (x)$ at $x$ if the set ${\partial}^{us}_D h (x) \subseteq \mathbb{R}^{n}$ is closed and satisfies
 \begin{equation}\label{eq2}
  h^+ (x, d) \leq \underset{x^* \in {\partial}^{us}_D h (x)}{\mathop{\sup}}  \langle x^*, d\rangle, ~ \forall ~ d \in D.
 \end{equation}
 \end{itemize}
Since $h^- (x,d) \leq h^{+} (x,d)$ for all $d \in D$, a directional upper semiregular convexificator is also a directional upper convexificator of $f$ at $x$.

If $h$ is continuous at a point $x$, then the directional convexificator 
collapses to the usual convexificator and the set of all possible directions $D$ is
$\mathbb{R}^{n}$, but these notions are not equivalent (see \cite[Example 2.10]{gadhi2022optimality}). Furthermore, the directional upper (resp. upper semi-regular) convexificator coincides with the standard upper 
(resp. upper semi-regular) convexificator. Moreover, if the inequality (\ref{eq2}) 
becomes equality for every $d \in D$, then ${\partial}^{us}_D h (x)$ is called a directional upper regular convexificator of $h$ at $x$. For further details we refer to \cite[Definition 3]{dempe2015necessary}.

For establishing the optimality conditions presented in Section 3, we introduce the following assumption: 
\begin{itemize}
 \item[$(H_0)$] Let the function $h: \mathbb{R}^n \rightarrow \mathbb{R}$ be lsc and the set $D(x)$ denote all continuity directions of $h$ at $x$. Then
$$ \exists\, \delta>0, \forall\, d \in D(x), \forall\, \bar{x} \in (x, x+\delta d) : \lim_{t \searrow 0} h (\bar{x}+td) =\lim_{t \searrow 0} h (\bar{x} -td) = h (\bar{x}). $$
\end{itemize}


The following outcome presents a chain rule for the directional semi-regular convexificator. Since its proof is similar to the one of  \cite[Proposition 1]{dempe2015necessary}, it is omitted. As usual, we write usc for upper semicontinuous.

\begin{lemma}\label{pro1}
 Let $g :\mathbb{R}^{n} \rightarrow \mathbb{R}$ be a continuous function and $h = (h_1, \ldots, h_n)$ be such that every $h_i: \mathbb{R}^{p} \rightarrow \mathbb{R}$ is lsc. Suppose that for every $i= 1,\dots,n,$ $h_i$ admits a bounded directional semi-regular convexificator ${\partial}^*_{D_i} h_i (\check{x})$ at $\check{x}$ with continuity direction $D_i(\check{x})$ and that $g$ admits a bounded semi-regular convexificator ${\partial}^* g (h (\check{x})) $ at $h (\check{x})$. Furthermore, suppose that assumption $(H_0)$ is satisfied for each $h_i$ at $\check{x}$. If for every $i= 1,\dots,n,$ the directional semi-regular convexificator ${\partial}^*_{D_i} h_i (\cdot)$ is usc at $\check{x}$ on $U(\check{x}) \cap (\{\check{x}\} + \bar{D})$ (where $U(\check{x})$ denotes an open neighborhood of $\check{x}$ and $\bar{D} = \bigcap_{i=1}^n D_i(\check{x}) \neq \{\textbf{0}_p\}$) and ${\partial}^* g (\cdot)$ is usc at $h(\check{x})$, then 
 \begin{align*}
 {\partial}^*_{\bar{D}} (g \circ h) (\check{x}) & = {\partial}^* g \Big(h (\check{x})\Big) \Big({\partial}^*_{D_1} h_1 (\check{x}), \dots, {\partial}^*_{D_n} h_n (\check{x})\Big) \\ 
 & =\Big\{ \sum_{i=1}^{n} a_ih_i :~ a \in {\partial}^* g \big(h (\check{x})\big), h_i \in {\partial}^*_{D_i} h_i (\check{x}) \Big\},
 \end{align*}
 is a directional semi-regular convexificator of $g \circ h$ at $\check{x}$.
\end{lemma}

We also recall the following classical result of convex analysis.

\begin{lemma}\label{lm1} {\rm (see \cite{Gadhi2021})}
 Assume that $ \mathcal{B} \subset \mathbb{R}^{n}$ be a nonempty, convex 
 and compact set and $\mathcal{A}\subset \mathbb{R}^{n}$ be a convex 
 cone and that 
 $$\underset{e \in \mathcal{B}} {\mathop{\sup}}\  \langle e, d \rangle \geq 0, ~ \forall ~ d \in \mathcal{A}^{\circ}.$$
 Then $0 \in \mathcal{B} + cl\mathcal{A}$.
\end{lemma}

\section{Necessary optimality conditions}\label{sec:03}

In order to provide necessary optimality conditions for a solution to be local 
weak Pareto solution of problem $($\text{A}$)$ using directional convexificators, 
we introduce a single-level problem.

Let $x \in \mathbb{R}^{n_1}$ and $Y(x) = \{ y \in \mathbb{R}^{n_2}: \, \phi_s(x,y) \leq 0, ~ \forall ~ s \in S\}$ be the feasible region of the lower level problem $(\text{A}_{x})$. In what follows, we assume that the set-valued map $Y$ is uniformly bounded around $\check{x}$, that is, there exists a bounded neighborhood $\mathbb{U} (\check{x}, \check{y})$ of $(\check{x}, \check{y})$ such that $\bigcup_{x \in \mathbb{U}} Y(x)$ is bounded. Here $\mathbb{U}$ is defined as
\begin{equation*}
 \mathbb{U} := \{x \in \mathbb{R}^{n_1}: \exists\ y \in \mathbb{R}^{n_2} \ 
 \text{with} \ (x,y) \in \mathbb{U}(\check{x},\check{y}) \}.
\end{equation*}

Let $\bar{\mathcal{B}}_{\mathbb{R}^{n_2}}$ be the closed unit ball of
$\mathbb{R}^{n_2}$ and $\mathbb{U}_{\check{x}} := U^0 \cap \mathbb{U}$ 
(where $U^0$ is a neighborhood of $\check{x}$). Then  $cl\left(\bigcup_{x \in
\mathbb{U}_{\check{x}}} Y(x)\right)$ is bounded and $cl\left(\bigcup_{x \in \mathbb{U}_{\check{x}}} Y(x)\right)$ is com\-pact. Take
\begin{equation*}
 \Theta = cl\left(\bigcup_{x \in \mathbb{U}_{\check{x}}} Y (x) \right) 
 + \bar{\mathcal{B}}_{\mathbb{R}^{n_2}}.
\end{equation*}
Then $\Theta$ is a nonempty compact set which includes an open neighborhood of $cl\left(\bigcup_{x \in \mathbb{U}_{\check{x}}} Y(x)\right)$ by
\cite{gadhi2012necessary}. This allow us to establish a real-valued function $\Psi: \mathbb{R}^{n_1} \times \mathbb{R}^{n_2} 
\rightarrow  \mathbb{R}$ such that
\begin{equation*}
 \Psi(x,y) = \underset{z \in \Theta}{\max}\ \psi(x,y,z), ~ \forall ~ (x,y)\in \mathbb{R}^{n_1}\times \mathbb{R}^{n_2},
\end{equation*}
where $\psi : \mathbb{R}^{n_1}\times \mathbb{R}^{n_2}\times \Theta 
\rightarrow \mathbb{R}$ is given by
\begin{equation*}
 \psi (x, y, z) := \min\{f(x,y)-f(x,z),~ -\Delta_{(-\mathbb{R}_+^q)} (\phi_1 (x, z), 
 \dots, \phi_q(x,z))\}.
\end{equation*}

Following \cite[Lemmas 3.1, 3.2 and 3.3]{gadhi2012necessary}, we known that a local weak efficient solution of $(A)$ is a local weak efficient solution of the single-level problem
\begin{equation*}
	(\text{$A^*$}):\quad 
	\begin{cases}
		\underset{\qquad \ x, y}{\mathop{\mathbb{R}^{n}_+ - \min }} \ \frac{\mathcal{F}(x,y)}{\mathcal{G}(x,y)}= \left(\frac{\mathcal{F}_1(x,y)}{\mathcal{G}_1(x,y)},  \dots,\frac{\mathcal{F}_n(x,y)}{ \mathcal{G}_n(x,y)}\right)\\
		s. \ t. ~~ (x,y) \in E
	\end{cases}
\end{equation*}\
where
\begin{equation*}
 E = 
  \begin{cases}
  (x, y) \in \mathbb{R}^{n_1} \times \mathbb{R}^{n_2}, \\
  {\mathcal{H}_j} (x,y) \leq 0, ~ \forall ~ j \in J = \{1,\ldots,p\}, \\
  ~ {\phi_s} (x,y) \leq 0, ~ \forall ~ s \in S = \{1,\ldots,q\}, \\
  ~\, \Psi(x,y) \leq 0. \\ 
 \end{cases}
\end{equation*}

Following \cite[Remark 3.1]{gadhi2012necessary}, we analyze the existence of solution for problem \text{(A)} in the next remark. To that end, we define $\Phi_k (\check{x}, \check{y}) = \frac{\mathcal{F}_k (\check{x}, \check{y})} {\mathcal{G}_k (\check{x}, \check{y})}$ and $\varphi_k (x,y) = \mathcal{F}_k (x, y) - \Phi_k (\check{x}, \check{y}) \mathcal{G}_k (x, y)$ for every $k= 1,\ldots,n$.

\begin{remark}
 If the following assumptions holds:
 \begin{itemize}
  \item[$(T_1)$] $\Phi_k$, $k=1, \dots, n$, $\mathcal{H}_j$, $j \in J$, $\phi_s$, $s \in S$ and $\Psi$ are lsc on $\mathbb{R}^{n_1}\times \mathbb{R}^{n_2}$.
   		
 \item[$(T_2)$] Problem (\text{$A^*$}) has at least one solution and its feasible set is bounded.
\end{itemize}
 then problem \text{(A)} has at least one optimal solution (see 
 \cite[Remark 3.1]{gadhi2012necessary}). In particular, under assumptions $(T_{1})$ and $(T_{2})$, the feasible set $E$ is a nonempty compact set, and the function $\Phi_k$ is lsc for all $k=1,\dots, n$.
\end{remark}

Now, let us consider a nonempty cone $D$ of $\mathbb{R}^{n_1} \times 
\mathbb{R}^{n_2}$ such that 
$$ D =  \left( \bigcap_{k=1}^{n} D_{\Phi_k} \right) \cap \left( \bigcap_{j \in J} D_{\mathcal{H}_j} \right) \cap \left(\bigcap_{s \in S} D_{\phi_s} \right) \cap \, D_{\Psi},$$
where $D_{\Phi_k}$, $D_{\mathcal{H}_j}$, $D_{\phi_s}$, and 
$D_{\Psi}$ are the sets of all continuity directions of $\Phi_k$, $\mathcal{H}_j$, 
$\phi_s$ and $\Psi$ at ($\check{x},\check{y}$), respectively. 

Let $(\check{x},\check{y}) \in E$ and the sets
\begin{align*}
 & J_0 (\check{x},\check{y})= \{j \in J : \mathcal{H}_j 
 (\check{x},\check{y})=0\}, \\
 & S_0(\check{x},\check{y})= \{s \in S : \phi_s (\check{x},\check{y})=0\}.
\end{align*}

Then we introduce the following constraint qualification.

\begin{definition}
 \textbf{[${\partial}_D$-nonsmooth ACQ]:} Assuming that all the functions in $E$ admit (DUCF) ${\partial}^{u}_D \mathcal{H}_j (\check{x}, \check{y})$ (with $j \in J$), ${\partial}^{u}_D \phi_s (\check{x}, \check{y})$ (with $s \in S$) at $(\check{x}, \check{y}) \in E$ and ${\partial}^{u}_D \Psi (\check{x}, \check{y})$ at
 $(\check{x}, \check{y}) \in E$. We introduce the following notation: 
 \begin{align*} 
  & \hspace{1.0cm} \, \Xi (\check{x}, \check{y}) := \mathbb{C} (\check{x}, 
 \check{y}) \cup N_D (0_{n_1+n_2}), \\
 & ~ T_{D} (E, (\check{x}, \check{y})) := T (E, (\check{x}, \check{y})) \cap D, \\ 
 & W_D (E, (\check{x}, \check{y})) := W (E, (\check{x}, \check{y})) \cap D,
\end{align*}
where
\begin{equation*}
 \mathbb{C} (\check{x},\check{y}) = \left(\bigcup_{j \in J_0 (\check{x}, \check{y})} {\partial}^{u}_D \mathcal{H}_j (\check{x}, 
 \check{y}) \right) \cup \left( \bigcup_{j \in S_0 (\check{x}, \check{y})}  
 {\partial}^{u}_D \phi_s (\check{x}, \check{y}) \right) \cup \ {\partial}^{u}_D 
 \Psi (\check{x},\check{y}).
\end{equation*}
It is said that the generalized ${\partial}_D$-nonsmooth ACQ holds at $(\check{x},\check{y}) \in E$ if 
\begin{equation*}
 [\Xi (\check{x}, \check{y})]^{\circ}  \subseteq T_{D} (E, (\check{x}, \check{y})).
\end{equation*} 
\end{definition}

In the following remark, we analyze the importance of the cone $N_{D} (0_{n_1+n_2})$ in the previous definition.

\begin{remark} 
 Note that in our definition of the cone $\Xi$, we have included $N_{D} (0_{n_1+n_2})$, which is not standard in the literature, because its play a fundamental role in our proof. Furthermore, we emphasize that when the functions \(\mathcal{H}_j\) \((j \in J)\), \(\phi_s\) \((s \in S)\), and \(\Psi\) are continuous at \((\check{x}, \check{y})\), we have \(D = \mathbb{R}^{n_1} \times \mathbb{R}^{n_2}\). Consequently, the directional upper convexificators at \((\check{x}, \check{y})\) (namely, \({\partial}^{u}_D \mathcal{H}_j (\check{x}, \check{y})\), \({\partial}^{u}_D \phi_s (\check{x}, \check{y})\), and \({\partial}^{u}_D \Psi (\check{x}, \check{y})\)), coincide with their standard upper convexificators \({\partial}^{u} \mathcal{H}_j (\check{x}, \check{y})\), \({\partial}^{u} \phi_s (\check{x}, \check{y})\), and \({\partial}^{u} \Psi (\check{x}, \check{y})\).

 In this case, the set \(\Xi (\check{x}, \check{y})\) reduces to \(\mathbb{C} (\check{x}, \check{y})\), and one recovers the nonsmooth Abadie-type constraint qualification introduced in \cite[De\-fi\-ni\-tion 4.1]{lafhim2018necessary}.
\end{remark}

Our first main result, which provides necessary optimality conditions for a local weak Pareto solution for problem \text{(A)}, is given below.

\begin{theorem}\label{thm N1}
 Let $(\check{x}, \check{y}) \in E$ be a local weak Pareto solution for problem \text{(A)}. Suppose that ${\partial}_D$-nonsmooth ACQ holds at $(\check{x}, \check{y})$ and $\mathcal{F}_k,\ (-\mathcal{G}_k)$ admit bounded
$($DUSRCF$)$ ${\partial}^{us}_D \mathcal{F}_k (\check{x},\check{y})$, ${\partial}^{us}_D (-\mathcal{G}_k) (\check{x},\check{y})$ at 
$(\check{x}, \check{y})$, with $k = 1,\dots,n$, respectively, $\mathcal{H}_j \ (j \in J =\{1,\dots,p\} ),\ \phi_s \ (s \in S = \{1,\dots,q\})$ and $\Psi$ admit (DUCF) ${\partial}^{u}_D \mathcal{H}_j 
(\check{x},\check{y}),\ {\partial}^{u}_D \phi_s (\check{x},\check{y})$, ${\partial}^{u}_D \Psi (\check{x},\check{y})$ at $(\check{x}, \check{y})$, 
respectively and that ${\partial}^{us}_D \mathcal{F}_k (\check{x},\check{y})$, 
${\partial}^{us}_D (-\mathcal{G}_k) (\check{x},\check{y})$ ${\partial}^{u}_D \mathcal{H}_j 
(\check{x},\check{y}),\ {\partial}^{u}_D \phi_s (\check{x},\check{y})$, ${\partial}^{u}_D \Psi (\check{x},\check{y})$ are usc at $(\check{x},\check{y})$. If the set $E$ is locally star-shaped at 
$(\check{x}, \check{y})$ and $pos~\Xi (\check{x},\check{y})$ is closed, then 
there exists $\delta^* = (\xi^*, \tau^*, \rho^*, \eta^*) \in \mathbb{R}_+^{(n+p+q+1)}$, with $\xi^* \neq 0_{\mathbb{R}^n}$, such that 
\begin{align}
 & (0,0) \in \sum_{k=1}^{n} \xi^*_k  \partial^{us}_D \varphi_k (x,y) + \sum_{j=1}^{p} \tau_j^* \ \partial^{u}_D \mathcal{H}_j (\check{x}, \check{y})+ \sum_{s=1}^{q} \rho^*_s \ {\partial}^{u}_D \phi_s (\check{x}, \check{y})  \notag \\
 & \hspace{1.0cm} + \eta^*\ {\partial}^{u}_D \Psi (\check{x}, \check{y}) + N_D (0_{n_1+n_2}). \label{eq3} \\
 & \tau_j^* \mathcal{H}_j (\check{x}, \check{y}) =0, ~ \forall ~ j \in J, \notag \\
 & \rho^*_s \phi_s (\check{x}, \check{y}) = 0, ~ \forall ~ s \in S. \notag
 \end{align}
\end{theorem}

\begin{proof} 
 Since $(\check{x}, \check{y}) \in E$ is a local weak Pareto solution for problem \text{(A)}, it is a local weak Pareto solution of (\text{$A^*$}). Hence, there exists a neighborhood $\mathbb{U}$ of $(\check{x},\check{y})$ such that
 \begin{equation*}
  \frac{\mathcal{F}_k (x, y)}{\mathcal{G}_k (x, y)} - \frac{\mathcal{F}_k (\check{x}, \check{y})}{\mathcal{G}_k (\check{x}, \check{y})}\ \notin - int\ {\mathbb{R}^n_+}, ~~ \forall ~ (x, y) \in E \cap  \mathbb{U}.
 \end{equation*}

 We claim that $(\check{x},\check{y})$  is a local weak Pareto solution of the  problem
  \begin{equation*}
  (\text{P}_{*})\quad 
   \begin{cases}
    \underset{x, y}{\mathop{\min }} \bigg(\mathcal{F}_1 (x, y) - \Phi_1 (\check{x}, \check{y}) \mathcal{G}_1 (x, y), \ldots, \mathcal{F}_n (x, y) - \Phi_n (\check{x}, \check{y}) \mathcal{G}_n (x, y) \bigg) \\
    s. \ t. \quad  (x, y) \in E.
  \end{cases}
 \end{equation*}
 Indeed, suppose for the contrary that there exists a point $(x_0, y_0) \in E \cap V$ with
 \begin{align*}
  \mathcal{F}_k (x_0, y_0) - \Phi_k (\check{x}, \check{y}) \mathcal{G}_k 
  (x_0, y_0) - \mathcal{F}_k (\check{x}, \check{y}) - \Phi_k (\check{x}, 
  \check{y}) \mathcal{G}_k (\check{x}, \check{y}) < 0.
 \end{align*}
 Since $\mathcal{F}_k (\check{x}, \check{y}) - \Phi_k (\check{x}, \check{y})
 \mathcal{G}_k (\check{x}, \check{y}) = 0$ for every $k= 1,\dots,n$, we have
\begin{align*}
 \mathcal{F}_k (x_0, y_0) < \Phi_k (\check{x}, \check{y})
 \mathcal{G}_k (x_0, y_0) < 0
 & \Longleftrightarrow \, \frac{\mathcal{F}_k (x,y)}{\mathcal{G}_k (x,y)} < \frac{\mathcal{F}_k (\check{x}, \check{y} )}{\mathcal{G}_k 
 (\check{x}, \check{y} )}, ~ \forall ~ k= 1,\ldots,n,
\end{align*}
which contradicts that $(\check{x}, \check{y}) \in E$ is a local weak Pareto solution for problem \text{(A)}, thus  $(\check{x}, \check{y})$ is a local weak Pareto solution of $(P_*)$.

Now, since $\Phi (\check{x}, \check{y}) = \frac{\mathcal{F} (\check{x}, 
 \check{y})}{\mathcal{G} (\check{x}, \check{y})}$ and $\varphi (x, y) = 
 \mathcal{F} (x, y) - \Phi (\check{x}, \check{y}) \mathcal{G}(x,y)$, we set 
 $F (x, y) = \varphi (x,y) - \varphi (\check{x}, \check{y})$. Let us 
 consider the problem
 \begin{equation*}
  \begin{cases}
   \underset{x, y}{\mathop{ \min}} \ \mathcal{K} (x, y) = \Delta_{-int 
   \mathbb{R}_+^n} \circ F (x, y) \\
   s. \ t. \quad (x,y) \in E.
  \end{cases}
 \end{equation*}
Then $(\check{x}, \check{y})$ is a local solution. 
Take $d = (d_1, d_2) \in cl\,W_D (E, (\check{x}, \check{y}))$. Then, there exist $d_n \rightarrow d$ as $n \rightarrow \infty$ such that $d_n \in W_D (E, (\check{x}, \check{y}))$, i.e., $d_n \in W_D (E, (\check{x}, \check{y})) \cap D$. Hence, we can find $t_n \downarrow 0$ such that  $(\check{x}, \check{y}) + t_n d_n \in E$. Furthermore, since $(\check{x}, \check{y})$ is local minimum of $\mathcal{K}$ over $E$, we have
 \begin{align*}
  \frac{\mathcal{K}((\check{x}, \check{y}) + t_n d_n) - \mathcal{K} 
  (\check{x}, \check{y})}{t_n}  \geq 0 & \Longrightarrow \, 
  \underset{n}{\limsup} \frac{\mathcal{K}((\check{x}, \check{y}) + t_n d_n) 
  - \mathcal{K} (\check{x},\check{y})}{t_n}  \geq 0 \\
  & \Longrightarrow \, \mathcal{K}^+((\check{x}, \check{y}), d) \geq 0.
 \end{align*}
Using the upper semi-regularity of ${\partial}^{us}_D\mathcal{K} 
(\check{x},\check{y})$ at $(\check{x},\check{y})$, we have
\begin{align*}
  & \hspace{0.9cm} \underset{\zeta \in {\partial}^{us}_D \mathcal{K} 
  (\check{x}, \check{y})}{\sup} \langle \zeta, d \rangle  \geq 0, ~ \forall ~ 
 d \in cl\,W_D (E, (\check{x}, \check{y})) \\
 & \Longrightarrow \, \underset{\zeta \in co{\partial}^{us}_D \mathcal{K} 
 (\check{x}, \check{y})}{\sup} \langle \zeta, d \rangle  \geq 0, ~ \forall ~  d \in cl\,W_D (E, (\check{x}, \check{y})),
\end{align*}
 and since $E$ is locally star-shaped at $(\check{x}, \check{y})$,
\begin{align*}
 T_{D} (E, (\check{x}, \check{y})) = cl\,W_D (E, (\check{x}, \check{y})) \Longrightarrow \, \underset{\zeta \in co{\partial}^{us}_D \mathcal{K} (\check{x},\check{y})}{\sup} \langle \zeta, d \rangle \geq 0, ~ \forall ~  d \in T_{D} (E, (\check{x}, \check{y})).
\end{align*}
		
It follows from ${\partial}_D$-nonsmooth ACQ assumption at 
$(\check{x}, \check{y})$ that
 \begin{align*}
 [\Xi (\check{x}, \check{y})]^{\circ} \subseteq T_{D} (E, (\check{x}, \check{y})) \, \Longleftrightarrow \,
 \underset{\zeta \in co{\partial}^{us}_D \mathcal{K} (\check{x}, 
 \check{y})}{\sup} \langle \zeta, d \rangle  \geq 0, ~ \forall~  d \in \Xi 
 (\check{x},\check{y})^{\circ}.
\end{align*}
Since $\Xi(\check{x},\check{y}) \subseteq pos~\Xi(\check{x},\check{y})$, we 
obtain 
\begin{equation*}
 \underset{\zeta \in co{\partial}^{us}_D \mathcal{K}(\check{x}, 
 \check{y})}{\sup} \langle \zeta, d \rangle  \geq 0, ~ \forall ~  d \in 
[pos~\Xi (\check{x},\check{y})]^{\circ}.
\end{equation*}
In virtue of the closedness of ${\partial}^{us}_D \mathcal{K} (\check{x}, 
\check{y})$, the set $co \, {\partial}^{us}_D \mathcal{K} (\check{x}, \check{y})$ 
 is compact (see \cite[Theorem 1.4.3]{hiriart2004fundamentals}). Hence, it 
 follows from Lemma \ref{lm1} with $A = pos\,\Xi (\check{x},\check{y})$ and 
$B = co\,{\partial}^{us}_D \mathcal{K}(\check{x},\check{y})$ that
$$(0,0) \in co~{\partial}^{us}_D \mathcal{K}(\check{x},\check{y}) + 
cl\,pos\,\Xi (\check{x},\check{y}).$$
		
But $pos \, \Xi (\check{x}, \check{y})$ is closed too, thus
$$(0, 0) \in co\,{\partial}^{us}_D \mathcal{K} (\check{x}, \check{y}) +  
pos\,\mathbb{C} (\check{x}, \check{y}) + pos \, N_D(0_{n_1+n_2}),$$
and by the chain rule property in Lemma \ref{pro1}, there exists $\xi^* \in {\partial}^{*} \Delta_{-int\mathbb({R}_+^n)}(0)$ such that
$$(0, 0) \in co \, \Big(\xi^* \circ \big(\partial^{us}_D \varphi_1 (\check{x}, 
\check{y}),  \dots, \partial^{us}_D \varphi_n (\check{x}, \check{y}) \big) \Big) 
+ pos\,\mathbb{C} (\check{x}, \check{y}) + pos \, N_D (0_{n_1 + n_2}).$$
Since $N_D (0_{n_1 + n_2})$ is a convex cone, $pos \, N_D (0_{n_1+n_2}) 
=  N_D (0_{n_1+n_2})$. Then,	
\begin{align*}
 & (0,0) \in \xi^* \circ \Big((\partial^{us}_D \mathcal{F}_1 (\check{x}, \check{y}) + \Phi_1 (\check{x},\check{y}) \partial^{us}_D (-\mathcal{G}_1 ) (\check{x}, \check{y})), \ldots, (\partial^{us}_D \mathcal{F}_n 
 (\check{x},\check{y}) + \Phi_n (\check{x}, \check{y}) \partial^{us}_D
 (-\mathcal{G}_n) (\check{x}, \check{y})) \Big) \\ 
 & + pos\,\Bigg( \left( \bigcup_{j \in J_0 (\check{x}, \check{y})} {\partial}^{u}_D \mathcal{H}_j (\check{x}, \check{y}) \right) \cup \left( \bigcup_{j \in S_0 (\check{x}, \check{y})} {\partial}^{u}_D \phi_s (\check{x}, \check{y}) \right) \cup \ {\partial}^{u}_D \Psi (\check{x}, \check{y}) \Bigg) + N_D (0_{n_1 + n_2}).
\end{align*}
Then, there exists $\delta^* = (\xi^*, \tau^*, \rho^*, \eta^*) \in \mathbb{R}_+^{(n + p + q + 1)}$, with $\xi^* \neq 0_{\mathbb{R}^n}$, 
such that
\begin{align*}
 & (0, 0) \in \sum_{k=1}^{n}\xi^*_k  \big(\partial^{us}_D \mathcal{F}_k 
 (\check{x}, \check{y}) + \Phi_k (\check{x}, \check{y}) \big( \partial^{us}_D
 (-\mathcal{G}_k) (\check{x}, \check{y}) \big) + \sum_{j=1}^{p} \tau_j^* \
 \big(\partial^{u}_D \mathcal{H}_j (\check{x}, \check{y}) \\
 & \hspace{1.0cm} + \sum_{s=1}^{q} \rho^*_s \ {\partial}^{u}_D \phi_s (\check{x}, \check{y}) + \eta^*\ {\partial}^{u}_D \Psi (\check{x}, \check{y}) + N_D (0_{n_1 + n_2}), \\
 & \tau_j^* \mathcal{H}_j (\check{x}, \check{y}) = 0, ~ \forall ~ j \in J, \\
 & \, \rho^*_s \phi_s (\check{x}, \check{y}) = 0, ~ \forall ~ s \in S.
\end{align*}
Finally, since $\Delta_{(-\text{int} \mathbb{R}^{n}_+)}(.)$  is convex and
$\Delta_{-\text{int} (\mathbb{R}^{n}_+)}(0) = 0$, we deduce that 
$\Delta_{-\text{int} (\mathbb{R}^{n}_+)}(\xi) \geq \langle \xi^*, ~ \xi \rangle$.	
Hence, for all $\xi \in {-\text{int} \mathbb{R}^{n}_+}$, we have		
$$ \langle \xi^*,~ \xi \rangle \leq \Delta_{-\text{int} (\mathbb{R}^{n}_+)}(\xi) = -d(\xi,~ \mathbb{R}^{n} \backslash - \text{int} \mathbb{R}^{n}_+) \leq 0, $$
i.e., $\xi^* \in (-\mathbb{R}^n_+)^\circ$, thus $\xi^* \neq 0$ by Lemma \ref{prp1}. Then $\xi^* \in (-\mathbb{R}^n_+)^\circ \backslash \{0\}$, and the result follows by taking $\tau_j^*=0$ for $j \notin J_0 (\check{x},\check{y})$ and $\rho^*_s =0 $ for $s \notin S_0 (\check{x}, \check{y})$.
\end{proof}

\begin{remark}
Theorem \ref{thm N1} remains applicable even when the data are discontinuous, unlike the results in \cite[Theorem 1]{Babahadda2006}, \cite[Theorem 3.1]{gadhi2012necessary}, and \cite[Theorem 4.2]{lafhim2018necessary}. Furthermore, because the set of continuity directions for a given function is not necessarily convex or closed, Theorem \ref{thm N1} holds in this general context, whereas \cite[Theorem 1]{dempe2015necessary} does not.
\end{remark}

The following example illustrates the necessary optimality conditions from Theorem~\ref{thm N1} for a nondifferentiable bilevel problem.

\begin{example}
 Let us consider the multiobjective bilevel optimization problem:
 \begin{equation*}
  \text{$(Q_{1})$}\quad 
  \begin{cases}
   \underset{\qquad \ x, y}{\mathop{\mathbb{R}^{2}_+ - min }}\ \left( \frac{\mathcal{F}_1 (x, y)}{\mathcal{G}_1 (x, y)},  \frac{\mathcal{F}_2 (x, y)}{ \mathcal{G}_2 (x,y)} \right)  \\
   s.  t. ~ \mathcal{H}_1(x,y)=xy^2-x \leq 0 ,\\
   \qquad y \in \mathfrak{B} (x),
  \end{cases}
\end{equation*}\
where $\mathcal{F}_1 (x, y) = 6x-5y + \frac{5}{2}$, $\mathcal{G}_1 (x, y) =
x - \frac{3}{5}y + \frac{1}{2}$,  
$$
\mathcal{F}_2 (x,y) = \left\{
\begin{array}{ccl}
 3x + \frac{1}{2}y+1 & {\rm if} & x \geq 0, \, y \geq 0,\\
 3 + \vert y\vert & {\rm if} & x\geq 0, \, y < 0, \\
 x^2+1 & {\rm if} & x < 0, 
\end{array}
\right.
$$
and
$$
\mathcal{G}_2 (x,y) = \left\{
\begin{array}{ccl}
 x - \frac{1}{4}y + \frac{1}{2} & {\rm if} & x\geq 0, \\
 y+1 & {\rm if} & x < 0, 
\end{array}
\right.
$$
and for any $x \in \mathbb{R}$, $\mathfrak{B} (x)$ is the set of the optimal 
solutions to the following optimization problem
\begin{equation*}	\text{$((Q_{1})_{x})$}\quad 
		\begin{cases}
			\underset{ y}{\mathop{ min }} \ \ f (x,y)= x^2+x^{\frac{2}{3}}-y^3\\
			s. t.  \quad \phi_1(x,y)=y\leq 0,\\
			~~~~	\quad \phi_2(x,y)=-x^2-y\leq 0.
		\end{cases}
\end{equation*}
Note that $(\check{x},\check{y})=(0,0)$ is a feasible solution for problem \text{$(Q_{1})$} with $E = \mathbb{R}^ + \times \{0\}$, $\mathfrak{B} (x) = \{0\}$ and $\Psi(x,y)= \min \{-y^3,~ 0\}$.
	
Clearly, $\Phi_1(0,0) = 5$, $\Phi_2(0,0) = 2$, $\varphi_1(x,y)= x-2y$ and
$$ 
\varphi_2 (x,y) = \left\{
\begin{array}{ccl}
 x + y & {\rm if} & x\geq 0, \, y \geq 0, \\
 2 + \vert y\vert + \frac{1}{2}y-2x & {\rm if} & x\geq 0, \, y < 0, \\
  x^2 - 2y-1 & {\rm if} & x < 0.
\end{array}
\right.
$$
Since $D_{\varphi_1} = \mathbb{R}\times \mathbb{R}$, $D_{\varphi_2} = \mathbb{R}^{+} \times \mathbb{R}^{+}$, and $D_{\mathcal{H}_1} = D_{\phi_1} = D_{\phi_2} = D_{\Psi} =  \mathbb{R} \times 
\mathbb{R}$, we have
$$D = \mathbb{R}^{+}\times \mathbb{R}^{+}~~ \text{and}~~ 
N_{D}(0,0)= \mathbb{R}^{-} \times \mathbb{R}^{-}.$$
	
On the other hand, observe that ${\partial}^{us}_D \varphi_1(\check{x},\check{y}) 
= \{(1,-2)\}$ and ${\partial}^{us}_D \varphi_2(\check{x},\check{y}) = 
\{(1,1)\}$ are bounded DUSRCF of $\varphi_1$ and $\varphi_2$, respectively, 
while ${\partial}^{u}_D \mathcal{H}_1 (\check{x}, \check{y}) = \{(-1, 0)\}$, 
${\partial}^{u}_D \phi_1 (\check{x},\check{y}) = \{(0,1)\}$, 
${\partial}^{u}_D \phi_2 (\check{x}, \check{y}) = \{(0, - 1)\}$ and
${\partial}^{u}_D\Psi(\check{x},\check{y})= \{(0,0)\}$ are DUCF of 
$\mathcal{H}_1$, $\phi_1$, $\phi_2$ and $\Psi$, respectively.
		
Since, $\mathbb{C} (\check{x}, \check{y}) = \{(-1, 0), (0,1), (0,-1), 
(0,0)\}$, we have
\begin{align*}
 \Xi (\check{x}, \check{y}) = \{(-1,0), (0,1), (0,-1), (0,0)\} \cup ( \mathbb{R}^{-} \times \mathbb{R}^{-}) \Longrightarrow \, \Xi (\check{x},\check{y})^{\circ} = \mathbb{R}^{+} \times \{0\}.
\end{align*}
As $T_{D} (E, (\check{x}, \check{y})) = \mathbb{R}^{+} \times \{0\}$, we infer that $[\Xi (\check{x}, \check{y})]^{\circ} \subseteq T_{D} (E, (\check{x}, \check{y}))$. Hence, ${\partial}^{u}_D$-nonsmooth ACQ holds at $(\check{x}, \check{y})$, and since 
$pos~\Xi (\check{x},\check{y})= \mathbb{R}^{-} \times \mathbb{R}$, the set
$pos~\Xi (\check{x},\check{y})$ is closed.

Finally, taking $\xi^*_1 = \frac{1}{2}$, $\xi^*_2 = \frac{3}{2}$, $\tau^*_1 = \frac{1}{4}$,  $\rho^*_1 = \frac{3}{4}$, $\rho^*_2 = \frac{1}{4}$, 
$\eta^* = \frac{2}{3}$ and $(-\frac{7}{4}, -1) \in N_{D}(0,0)$,
we obtain
\begin{multline*}
 (0, 0) \in 
 \xi^*_1 ~  {\partial}^{us}_D \varphi_1(\check{x},\check{y}) + 
 \xi^*_2~  {\partial}^{us}_D \varphi_2(\check{x},\check{y}) + \tau_1^*  
 \partial^{u}_D \mathcal{H}_1 (\check{x},\check{y}) + \rho^*_1 \ 
 {\partial}^{u}_D \phi_1 (\check{x},\check{y})\\
 +\rho^*_2 \ {\partial}^{u}_D \phi_2 (\check{x},\check{y}) + 
 \eta^*\ {\partial}^{u}_D \Psi (\check{x},\check{y})+~N_D(0,0).
\end{multline*}
	Therefore, $(0,0)$ satisfies the necessary optimality conditions of Theorem \ref{thm N1} for local weak Pareto solution of \text{$(Q_{1})$}.
\end{example}

\section{Sufficient optimality conditions}

In order to provide sufficient optimality conditions, we recall the following 
generalized convexity notions for nonsmooth functions (see \cite{el2022applying}).

\begin{definition}{\rm (\cite{el2022applying})} \label{def:genpseudoconvex}
 Let $\mathcal{F}: \mathbb{R}^{n_1} \times \mathbb{R}^{n_2} \rightarrow
 \mathbb{R}$ and $(\check{a}, \check{b}) \in \mathbb{R}^{n_1} \times
 \mathbb{R}^{n_2}$. Suppose that $\mathcal{F}$ possesses a directional 
 upper (semi-regular) convexificator $\partial^{u}_D \mathcal{F} 
 (\check{a}, \check{b})$. Then $\mathcal{F}$ is said to be:
 \begin{itemize}[label=$\diamond.$]
  \item[$(i)$] $\partial^{u}_D$-convex at $(\check{a},\check{b})$ if for all 
  $(a,b) \in \mathbb{R}^{n_1} \times \mathbb{R}^{n_2},$ where 
  $(a,b)-(\check{a},\check{b}) \in D$, we have
  \begin{equation*}
   \mathcal{F} (a, b) - \mathcal{F} (\check{a}, \check{b}) \geq \langle x^*, ~ (a,b) - (\check{a}, \check{b}) \rangle, ~ \forall ~ x^* \in \partial^{u}_D
   \mathcal{F}~(\check{a},\check{b}).
  \end{equation*}
		
  \item[$(ii)$] $\partial^{u}_D$- quasiconvex at $(\check{a},\check{b})$ if for all $(a,b) \in \mathbb{R}^{n_1} \times \mathbb{R}^{n_2}$, where $(a,b) -
  (\check{a},\check{b}) \in D$, we have
\begin{equation*}
 \mathcal{F} (a, b) - \mathcal{F} (\check{a}, \check{b}) \leq 0 \Longrightarrow \langle x^*, (a,b) - (\check{a}, \check{b}) \rangle \leq 0,  ~ \forall ~ x^* \in \partial^{u}_D \mathcal{F}~ (\check{a}, \check{b}).
 \end{equation*}

 \item[$(iii)$] $\partial^{u}_D$- pseudoconvex at $(\check{a}, \check{b})$ if 
  for all $(a,b) \in \mathbb{R}^{n_1} \times \mathbb{R}^{n_2}$, where $(a,b) 
 - (\check{a}, \check{b}) \in D \backslash \{0_{n_1+n_2}\}$, we have
 \begin{equation*}
  \mathcal{F} (a, b) - \mathcal{F} (\check{a}, \check{b}) < 0  \Longrightarrow \langle x^*,  (a,b) - (\check{a}, \check{b}) \rangle < 0, ~ \forall ~ x^* \in \partial^{u}_D \mathcal{F}~ (\check{a}, \check{b}).
 \end{equation*}
 \end{itemize}
\end{definition} 

For differentiable functions, pseudoconvexity was introduced in \cite{Manga1} for providing sufficient first-order optimality conditions for mathematical programming problems. In particular, the classical example of a differentiable fractional function in which its numerator is convex and its denominator is concave is pseudoconvex (see \cite[Theorem 3.2.10]{CM-Book}), i.e., Definition
\ref{def:genpseudoconvex} is closely related to our fractional problem.

In order to continue, we introduce the following sets:
\begin{align*}
 & \, J_0^{\neq} (\check{x}, \check{y}) \ = \{j \in J: \mathcal{H}_j (\check{x}, \check{y}) \neq 0\} \\
 & S_0^{\neq} (\check{x}, \check{y}) \ = \{s \in S: 
\phi_s (\check{x}, \check{y}) \neq 0\}.
\end{align*}

Our second main result, which provide sufficient optimality conditions for problem
\text{(A)}, is given below.

\begin{theorem}\label{thm S1}
 Let $(\check{x}, \check{y}) \in E$ be a feasible solution of \text{(A)}. Suppose that the functions $\varphi_k$ is $\partial^{u}_D$-pseudoconvex at $(\check{x}, \check{y})$ for every $k = 1,\dots,n$, the functions 
 $\mathcal{H}_j$, $j \in J_0 (\check{x}, \check{y})$, $\phi_s, s \in S_0 (\check{x}, \check{y})$ and $\Psi$ are $\partial^{u}_D$-quasiconvex at $(\check{x}, \check{y})$ and that $(x,y) - (\check{x}, \check{y}) \in D$. If there 
 exists $\delta^* = (\xi^*, \tau^*, \rho^*, \eta^*) \in \mathbb{R}_+^{
 (n+p+q+1)}$, with $\xi^* \neq 0_{\mathbb{R}^n}$, satisfying (\ref{eq3}), 
 then $(\check{x}, \check{y})$ is a weak Pareto solution of $($\text{A}$)$.
\end{theorem}

\begin{proof}
 Suppose for the contrary that $(\check{x}, \check{y})$ is not a weak 
 Pareto solution of ($\text{A}$). Hence $(\check{x}, \check{y})$ is not a 
 weak efficient solution of (\text{$A^*$}). Then there exist $(x, y) \in 
 E$  such that
 \begin{align*}
 \frac{\mathcal{F}_k(x,y) }{\mathcal{G}_k (x,y)} < \frac{\mathcal{F}_k (\check{x}, \check{y}) }{\mathcal{G}_k 
  (\check{x}, \check{y})} < 0  \Longleftrightarrow \, \mathcal{F}_k(x,y) - \Phi_k (\check{x}, \check{y}) \mathcal{G}_k (x,y) < 0, \, \forall ~ k= 1,\ldots,n. 
 \end{align*} 
 It follows from $\mathcal{F}_k (\check{x}, \check{y}) - \Phi_k (\check{x}, \check{y}) \mathcal{G}_k (\check{x}, \check{y}) = 0$ with $k= 1,\ldots,n$,
 that
 \begin{align*}
  & \mathcal{F}_k (x, y) - \Phi_k (\check{x}, \check{y}) \mathcal{G}_k (x, y) < \mathcal{F}_k (\check{x}, \check{y}) - \Phi_k (\check{x}, \check{y}) 
  \mathcal{G}_k (\check{x}, \check{y}) \\
  & \hspace{1.0cm} \Longleftrightarrow \, \varphi_k(x,y) <  \varphi_k 
  (\check{x},\check{y}),  ~ \forall ~ k= 1,\ldots,n.
 \end{align*}
Since for each $k=1,\ldots,n,~ \varphi_k$ is $\partial^{u}_D$-pseudoconvex at $(\check{x},\check{y})$ and $(x,y)-(\check{x},\check{y}) \in D$, we get 
$\langle \vartheta_k^*, (x, y) - (\check{x}, \check{y}) \rangle < 0$ for all $\vartheta_k^* \in \partial^{us}_D \varphi_k (\check{x}, \check{y})$.

By using \eqref{eq3}, there exist $\vartheta_k^{*'} \in \partial^{us}_D  \mathcal{F}_k(\check{x}, \check{y})$, $\vartheta_k^{*''} \in \partial^{us}_D 
(-\mathcal{G}_k) (\check{x}, \check{y})$, $\mu_j \in \partial^{u}_D \mathcal{H}_j (\check{x}, \check{y})$ for all $j \in J$, $\theta_s \in \partial^{u}_D \phi_s (\check{x}, \check{y})$ for all $s \in S$, $w^{*} \in \partial^{u}_D \Psi (\check{x}, \check{y})$ and $z^* \in N_D(0_{n_1+n_2})$ such that 
\begin{align}
 (0, 0) = & ~ \sum_{k=1}^{n} \xi^*_k  \big(\vartheta_k^{*'} + \Phi_k (\check{x},
 \check{y}) \vartheta_k^{*''} \big) + \sum_{j=1}^{p} \tau_j^* ~ \mu_j +
  \sum_{s=1}^{q} \rho^*_s \ \theta_s + \eta^*~ w^{*} + z^*,\label{eqn7} \\ 
  & \tau_j^* \mathcal{H}_j (\check{x}, \check{y}) =0, \ \rho^*_s \phi_s
  (\check{x},\check{y}) = 0, ~ \forall ~ j \in J, ~ \forall ~ s \in S. \notag
\end{align}

Since $\partial^{us}_D \varphi_k (\check{x}, \check{y}) = \partial^{us}_D 
F_k (\check{x}, \check{y}) + \Phi_k (\check{x}, \check{y})\ \partial^{us}_D 
(-G_k) (\check{x}, \check{y})$, we have $\vartheta_k^* = \big(\vartheta_k^{*'} 
+ \Phi_k (\check{x}, \check{y}) \vartheta_k^{*''} \big)$. Thus, 
$$\langle (\vartheta_k^{*'} + \Phi_k (\check{x}, \check{y}) \vartheta_k^{*''}), 
 (x, y) - (\check{x}, \check{y}) \rangle < 0, ~ \forall ~ k= 1,\ldots,n.$$
As $\xi^* = (\xi^*_1,\dots,\xi^*_n) \in (-\mathbb{R}^n_+)^\circ \backslash \{0\}$, we obtain
\begin{align}
 & \hspace{1.1cm} \Bigg\langle \sum_{k=1}^{n}\xi^*_k  \big( \vartheta_k^{*'} + \Phi_k (\check{x}, \check{y}) ~ \vartheta_k^{*''} \big),\ (x,y) - (\check{x}, \check{y}) \Bigg\rangle < 0, \notag \\
 & \overset{\eqref{eqn7}}{\Longrightarrow} \, \Bigg\langle \sum_{j=1}^{p} \tau_j^* ~ \mu_j + \sum_{s=1}^{q} \rho^*_s \ \theta_s + \eta^* ~ w^{*} + z^*, ~ (x,y) - (\check{x}, \check{y}) \Bigg\rangle > 0. \label{eq4}
\end{align}

Since $(x, y) \in E$, we have
 \begin{equation*} 
  \begin{cases}
   \mathcal{H}_j (x, y) \leq 0, ~ \forall ~ j \in J_0 (\check{x}, \check{y}), \\ 
   \phi_s (x, y) \leq 0, ~ \forall ~ s \in S_0 (\check{x}, 
   \check{y}), \\
   \Psi (x,y) \leq 0.
  \end{cases}
 \end{equation*}
 If $j \in J_0^{\neq} (\check{x}, \check{y})$, then $\tau_j^* = 0$, 
 while if $s \in S_0^{\neq} (\check{x}, \check{y})$, then $\rho^*_s = 0$. Hence,
 \begin{align}
  & \, \langle \tau_j^* \mu_j, (x, y) - (\check{x}, \check{y}) \rangle = 0, ~ \forall ~ j \in J_0^{\neq} (\check{x}, \check{y}), \label{eq5} \\
 & \langle \rho^*_s \theta_s,~ (x,y)- (\check{x}, \check{y}) \rangle = 0, ~ \forall ~ s \in S_0^{\neq}(\check{x}, \check{y}). \label{eq6}
\end{align}

 Note that $\Psi (\check{x}, \check{y}) = 0$, $\mathcal{H}_j (\check{x}, 
 \check{y}) = 0$ for all $j \in J_0 (\check{x}, \check{y})$ and $\phi_s (\check{x}, \check{y}) = 0$ for all $s \in S_0 (\check{x}, \check{y})$, thus
 \begin{equation*} 
  \begin{cases}
   \mathcal{H}_j (x, y) - \mathcal{H}_j (\check{x}, \check{y}) \leq 0, ~ \forall ~ j \in J_0 (\check{x}, \check{y}), \\ 
   ~~\phi_s (x, y) - \phi_s (\check{x}, \check{y}) \leq 0, ~ \forall ~ s \in S_0 (\check{x}, \check{y}), \\
   ~~~\Psi (x,y) -  \Psi (\check{x}, \check{y}) \leq 0.
  \end{cases}
 \end{equation*}

 Furthermore, by assumption, $\mathcal{H}_j$ for all $j \in J_0 (\check{x}, \check{y})$, $\phi_s$ for all $s \in S_0 (\check{x}, \check{y})$ and $\Psi$ are all $\partial^{u}_D$-qua\-si\-con\-vex at $(\check{x}, \check{y})$ and since $z^* \in N_D (0_{n_1 + n_2})$, we obtain
\begin{equation*} 
 \begin{cases}
  \langle \mu_j, (x, y) - (\check{x}, \check{y}) \rangle \leq 0, ~ \forall ~ 
  j \in J_0 (\check{x}, \check{y}), \\
  \langle \theta_s, (x, y) - (\check{x}, \check{y}) \rangle \leq 0, ~ \forall ~ 
  s \in S_0 (\check{x}, \check{y}), \\ 
  \, \langle w^{*}, (x, y) - (\check{x}, \check{y}) \rangle \leq 0, \\
  \, \langle z^*, (x, y) - (\check{x}, \check{y}) \rangle \leq 0. 
 \end{cases}
\end{equation*}
It follows from $\tau^*_j \geq 0$  for all $j \in J_0 (\check{x}, \check{y})$, $\rho^*_s \geq 0$  for all $s \in S_0 (\check{x}, \check{y})$ and $\eta^*\geq 0$ that
 \begin{equation} \label{eq7}
  \begin{cases}
   \langle \tau^*_j \mu_j, (x, y) - (\check{x}, \check{y}) \rangle  \leq 0, 
   ~ \forall ~ j \in J_0 (\check{x}, \check{y}), \\ 
   \, \langle \rho^*_s \theta_s, (x, y) - (\check{x}, \check{y}) \rangle \leq 0, 
   ~ \forall ~ s \in S_0 (\check{x}, \check{y}), \\ 
   ~ \langle \eta^* w^{*}, (x, y) - (\check{x}, \check{y}) \rangle \leq 0, \\ 
   ~~~~ \langle z^*, (x, y) - (\check{x}, \check{y}) \rangle \leq 0. 
   \end{cases}
  \end{equation}

Using equations \eqref{eq5}, \eqref{eq6} and \eqref{eq7}, and since $J = J_0 (\check{x}, \check{y}) \cup J_0^{\neq} (\check{x}, \check{y})$ and $S = S_0 (\check{x}, \check{y}) \cup S_0^{\neq} (\check{x}, \check{y})$, we obtain
\begin{equation*} 
 \begin{cases}
  \langle \tau^*_j \mu_j, (x,y)- (\check{x}, \check{y}) \rangle \leq 0, 
   ~ \forall ~ j \in J (\check{x}, \check{y}), \\ 
  \, \langle \rho^*_s \theta_s, (x, y) - (\check{x}, \check{y}) \rangle \leq 0, 
   ~ \forall ~ s \in S (\check{x}, \check{y}), \\
  \langle \eta^* w^{*}, (x,y)- (\check{x}, \check{y}) \rangle \leq 0, \\ 
  ~~~~ \langle z^*, (x, y) - (\check{x}, \check{y}) \rangle \leq 0. 
 \end{cases}
\end{equation*}
By adding the previous inequalities, we obtain a contradiction to \eqref{eq4}, 
and the result follows.
\end{proof}

\begin{remark}
Note that even when the data is non-continuous, Theorem \ref{thm S1} can be applied while \cite[Theorem 3.9]{gadhi2019sufficient} can not.
\end{remark}

In the example, we illustrate the sufficient optimality conditions established in Theorem \ref{thm S1}.

\begin{example}\label{ex2}
 Consider the multiobjective bilevel optimization problem:
 \begin{equation*}
  \text{$(Q_{1})$}\quad 
   \begin{cases}
    \underset{\qquad \ x, y}{\mathop{\mathbb{R}^{2}_+ - min }}\
    \Big[\frac{\mathcal{F}_1(x,y)}{\mathcal{G}_1(x,y)}, 
     ~\frac{\mathcal{F}_2(x,y)}{\mathcal{G}_2(x,y)}\Big]  \\
	s.  t. ~ \mathcal{H}_1(x,y) \leq 0 ,\\
   ~~~~~ \mathcal{H}_2(x,y) \leq 0, \\
	\qquad y \in \mathfrak{B} (x),
   \end{cases}
\end{equation*}
with $\mathcal{H}_1 (x, y) = - x - y$, $\mathcal{H}_2(x,y)= y$, 
$\mathcal{F}_2 (x, y) = x + 2y + \frac{3}{2}$, $\mathcal{G}_2 (x, y) = 
x + \frac{1}{3}y + \frac{1}{2}$,
$$ 
\mathcal{F}_1(x,y) = \left\{
\begin{array}{ccl}
 x^2+2x+3 & {\rm if} & x \geq 0, \, y \geq 0, \\
 3x^2+x+4y^2-2y+4  & {\rm if} & x \geq 0, \, y < 0, \\
 y^2+2 & {\rm if} & x < 0.
\end{array}
\right.
$$
$$ 
\mathcal{G}_1(x,y) = \left\{
\begin{array}{ccl}
 x-2y+3 & {\rm if} & x \geq 0, \\
 1-x & {\rm if} & x < 0.
\end{array}
\right.
$$
and $\mathfrak{B} (x)$, $x \in \mathbb{R}$, is the set of the optimal 
solutions of the problem \eqref{problem:ax}, with $\lvert S \rvert = 1,\, f (x, y) = {\sqrt[3]{x}} + \sqrt{x} + y^2 - y^3$ and $\phi_1 (x, y) = y^2-y$.

Now, note that $(\check{x}, \check{y}) = (0, 0)$ is a feasible solution for 
problem \text{$(Q_{1})$} with $E = \mathbb{R}^+ \times \{0\}$, $\mathfrak{B} (x) =\{0, 1\}$, $\Psi (x,y) = \min\{y^2 - y^3, 0\}$, $\Phi_1(0,0) = 1$, $\Phi_2(0,0) = 3$, $\varphi_2(x,y)= -2x+y$ and 
$$ 
\varphi_1 (x,y) = \left\{
\begin{array}{ccl}
 x^2 + x + 2y & {\rm if} & x \geq 0, \, y \geq 0, \\
 3x^2+4y^2+1  & {\rm if} & x \geq 0, \, y < 0, \\
 y^2+1+x & {\rm if} & x < 0.
\end{array}
\right.
$$	
Since $D_{\varphi_1} = \mathbb{R}^{+}\times \mathbb{R}^{+}$, 
$D_{\varphi_2} = \mathbb{R} \times \mathbb{R}$, $D_{\mathcal{H}_1} 
= D_{\mathcal{H}_2}= D_{\phi} = D_{\Psi} = \mathbb{R} \times
\mathbb{R}$, we have  $D = D (\check{x}, \check{y}) 
= \mathbb{R}^{+} \times \mathbb{R}^{+}$ and $N_{D} (0, 0) = 
\mathbb{R}^{-} \times \mathbb{R}^{-}$.
	
On the other hand, ${\partial}^{us}_D \varphi_1 (\check{x}, \check{y}) 
 = \{(1,2)\}$ and ${\partial}^{us}_D \varphi_2 (\check{x}, \check{y}) = 
\{(-2,1)\}$ are bounded DUSRCF of $\varphi_1$ and $\varphi_2$, respectively.
Furthermore, ${\partial}^{u}_D \mathcal{H}_1 (\check{x}, \check{y}) =
\{(-1,-1)\}$, ${\partial}^{u}_D \mathcal{H}_2 (\check{x}, \check{y}) = 
\{(0,1)\}$, ${\partial}^{u}_D \phi_1 (\check{x}, \check{y}) = 
\{(0,-1)\}$ and ${\partial}^{u}_D \Psi (\check{x}, \check{y}) = \{(0,0)\}$ 
are DUCF of $\mathcal{H}_1$, $\mathcal{H}_2$, $\phi$ and $\Psi$,
respectively.

We can check that $\varphi_1,\, \varphi_2$ are $\partial^{u}_D$-pseudoconvex at 
$(\check{x}, \check{y})$ and $\mathcal{H}_1$, $\mathcal{H}_2$, $\phi$ and 
$\Psi$ are $\partial^{u}_D$-quasiconvex at $(\check{x},\check{y})$. For instance, for all $(x, y) \in E$,

\begin{itemize}
    \item $ \big\langle {\partial}^{us}_D \varphi_1 (\check{x}, \check{y}),\, (x, y)-  (\check{x},\check{y})\big\rangle = x+2y \geq 0,$
$$ \Rightarrow \varphi_1 (x,y) = \left(\left\{
\begin{array}{ccl}
 x^2 + x + 2y & {\rm if} & x \geq 0, \, y \geq 0, \\
 3x^2+4y^2+1  & {\rm if} & x \geq 0, \, y < 0, \\
 y^2+1+x & {\rm if} & x < 0,
\end{array}
\right. \right) \geq \varphi_1 (\check{x}, \check{y}). $$

\item $\big\langle {\partial}^{us}_D \varphi_2 (\check{x}, \check{y}),\, (x, y)-  (\check{x},\check{y})\big\rangle = -2x+y \geq 0,$
$$ \Rightarrow \varphi_2 (x,y) = \left(-2x+y \right) \geq \varphi_1 (\check{x}, \check{y}). $$

\item $\mathcal{H}_1(x,y) \leq  \mathcal{H}_1 (\check{x}, \check{y}) \Rightarrow 
- x - y \leq 0,$
$$\Rightarrow \big\langle {\partial}^{u}_D \mathcal{H}_1 (\check{x}, \check{y}),\, (x, y)-  (\check{x},\check{y})\big\rangle = -x-y \leq 0.$$

\item $\mathcal{H}_2(x,y) \leq  \mathcal{H}_2 (\check{x}, \check{y}) \Rightarrow 
 y \leq 0,$
$$\Rightarrow \big\langle {\partial}^{u}_D \mathcal{H}_2 (\check{x}, \check{y}),\, (x, y)-  (\check{x},\check{y})\big\rangle = y \leq 0.$$

\item $\phi(x,y) \leq \phi (\check{x}, \check{y}) \Rightarrow 
 y^2-y \leq 0,$
$$\Rightarrow \big\langle {\partial}^{u}_D \phi (\check{x}, \check{y}),\, (x, y)-  (\check{x},\check{y})\big\rangle = -y \leq 0.$$

\item $\Psi(x,y) \leq \Psi (\check{x}, \check{y}) \Rightarrow 
 \min\{y^2 - y^3, 0\} \leq 0,$
$$\Rightarrow \big\langle {\partial}^{u}_D \phi (\check{x}, \check{y}),\, (x, y)-  (\check{x},\check{y})\big\rangle \leq 0.$$
\end{itemize}

Take $\delta^* = \big(\frac{3}{2},\frac{1}{4}, \frac{2}{5}, 1, 
\frac{7}{2}, \frac{1}{3} \big)$. Since $(-\frac{3}{5}, -\frac{7}{20}) \in N_{D} (0, 0)$, 
we have 
\begin{align*}
 (0, 0) \in \xi^*_1 {\partial}^{us}_D \varphi_1 (\check{x}, \check{y}) ~ & + 
 \xi^*_2  {\partial}^{us}_D \varphi_2(\check{x}, \check{y}) +  \tau_1^*
 \partial^{u}_D \mathcal{H}_1 (\check{x}, \check{y}) + \tau_2^* 
 \partial^{u}_D \mathcal{H}_2 (\check{x}, \check{y}) \\
 & + \rho^* \ {\partial}^{u}_D \phi_1 (\check{x}, \check{y}) + 
 \eta^*\ {\partial}^{u}_D \Psi (\check{x}, \check{y}) + N_D(0,0).
\end{align*}
Therefore, $(\check{x}, \check{y})$ is a weakly efficient solution of
\text{$(Q_{1})$} by Theorem \ref{thm S1}.
\end{example}

\section{Mond-Weir Dual Problem}\label{sec:05}

In this section, we develop both weak and strong Mond-Weir type duality results. To that end, let  us consider $(v, w) \in \mathbb{R}^{n_1} \times
\mathbb{R}^{n_2}$ and suppose that $\mathcal{F}_k$ and $(-\mathcal{G}_k)$ 
admit boun\-ded $($DUSRCF$)$ ${\partial}^{us}_D \mathcal{F}_k (v,w)$ 
and ${\partial}^{us}_D (-\mathcal{G}_k) (v, w)$ at $(v, w)$ for every 
$k = 1,\ldots,n$, respectively, and $\mathcal{H}_j$ (with $j \in J = \{1,\ldots,p\}$), $\phi_s$ (with $s \in S = \{1,\ldots,q\}$) and $\Psi$ admit (DUCF) ${\partial}^{u}_D \mathcal{H}_j (\cdot, \cdot)$, ${\partial}^{u}_D \phi_s (x, y)$ and ${\partial}^{u}_D \Psi (x, y)$ at $(v, w)$, respectively. 

Then the associated Mond-Weir dual problem for problem \text{(A)} (with solution set denoted by ${\Omega}$) is given by:
\begin{equation*}
	(\text{DA})\ 
\begin{cases}
 {\mathop{\mathbb{R}^{n}_+ - max }} \ \Phi (v, w) = \frac{\mathcal{F} (v, w)
 }{\mathcal{G} (v, w)} = \left(\frac{\mathcal{F}_1 (v, w)}{\mathcal{G}_1 
 (v, w)}, \dots, \frac{\mathcal{F}_n (v, w)}{\mathcal{G}_n (v, w)} \right) \\
 s.~t. \ \\
 (0, 0) \in \sum_{k=1}^{n}\xi^*_k \Big({\partial}^{us}_D \mathcal{F}_k
 (v, w) + \Phi_k (v, w){\partial}^{us}_D (-\mathcal{G}_k) (v, w) \Big) + N_D 
 (0_{n_1 + n_2}) \\
 \vspace{0.1cm}
 \hspace{1.0cm} + \sum_{j=1}^{p} \tau_j^* \ {\partial}^{u}_D \mathcal{H}_j 
 (v, w) + \sum_{s=1}^{q} \rho^*_s \ {\partial}^{u}_D \phi_s (v, w) 
 + \eta^*\ {\partial}^{u}_D \Psi (v, w) \\		
 \vspace{0.1cm}
 ~ \tau^*_j \mathcal{H}_j (v,w) \geq 0, ~ \forall ~ j \in J, \\  
 ~~ \rho^*_s \phi_s (v,w) \geq 0 ,~ \forall ~ s \in S, \\ 
 \vspace{0.1cm}
 ~~~ \eta^* \Psi (v, w) \geq 0, \\
 \vspace{0.1cm}
 ~ \Upsilon^* = (\xi^*_{1}, \dots, \xi^*_n, \tau^*_1, \dots, \tau^*_p, 
 \rho^*_1, \dots, \rho^*_q, \eta^*) \geq 0, \\
 ~ (\xi^*_1, \dots, \xi^*_n) \neq 
 (0, \dots, 0).
 \end{cases}
\end{equation*}

Sufficient conditions for weak duality are given in the next result.

\begin{theorem} {\rm \textbf{(Weak Duality):}} \label{thm D1}
 Let $(\bar{x}, \bar{y})$ and $(v, w, \Upsilon^*)$  be feasible points of problems \text{(A)} and \text{(DA)}, respectively. Suppose that $\big(\mathcal{F}_k + \Phi_k (v, w) (-\mathcal{G}_k)\big)$ is $\partial^{u}_D$-pseudoconvex at $(v, w)$  for every $k = 1,\dots,n$, $\mathcal{H}_j$ (with $j \in J_0 (\bar{x}, \bar{y})$), $\phi_s$ (with $s \in S_0 (\bar{x}, \bar{y})$) and $\Psi$ are $\partial^{u}_D$-quasiconvex at $(v,w)$, respectively, and that
 $(\bar{x}, \bar{y}) - (v,w) \in D$. Then 
 \begin{equation}
  \Phi(\bar{x}, \bar{y}) \nless \Phi(v,w).
 \end{equation}
\end{theorem}

\begin{proof}
 Suppose for the contrary that there exists a feasible point $(\bar{x}, \bar{y})$ of \text{(A)} and a feasible point $(v, w, \Upsilon^*)$ of \text{(DA)} such that 	
 \begin{align}
  & \Phi (\bar{x}, \bar{y}) < \Phi (v, w) \Longrightarrow ~ {\mathcal{F}_k}(\bar{x}, \bar{y}) + \Phi_k(v,w) (-\mathcal{G}_k) (\bar{x}, \bar{y}) < 0, ~ \forall ~ k \in \{1,\ldots,n\}. 
 \end{align}
 Thus ${\mathcal{F}_k}(\bar{x}, \bar{y}) + \Phi_k(v,w)(-\mathcal{G}_k) 
 (\bar{x}, \bar{y}) < 0$. Since $\mathcal{F}_k (v, w) + \Phi_k  (v, w) (-\mathcal{G}_k) (v, w) = 0$, with $k= 1,\ldots,n$, we have
 \begin{equation*}
  {\mathcal{F}_k} (\bar{x}, \bar{y}) + \Phi_k (v, w) (-\mathcal{G}_k )
  (\bar{x}, \bar{y}) < \mathcal{F}_k (v, w) + \Phi_k (v, w) (-\mathcal{G}_k )
  (v, w).
 \end{equation*}
 Thus, $(\bar{x}, \bar{y}) \neq (v,w)$. As $\xi := (\xi^*_1, \ldots ,\xi^*_n) 
 \geq 0$, with $\xi \neq 0$, we have
\begin{equation*}
 \sum_{k=1}^{n} \xi^*_k ({\mathcal{F}_k}(\bar{x}, \bar{y}) + \Phi_k (v,w) (-\mathcal{G}_k )(\bar{x}, \bar{y}) ) < \sum_{k=1}^{n} \xi^*_k ( \mathcal{F}_k (v,w) + \Phi_k (v,w) (-\mathcal{G}_k) (v, w)).   
\end{equation*}
	
Since $(v, w, \Upsilon^*) \in {\Omega}$, there exists $\vartheta_k^{*'} \in  \partial^{us}_D \mathcal{F}_k (v, w)$,  $\vartheta_k^{*''} \in \partial^{us}_D (-\mathcal{G}_k ) (v, w)$, $\mu_j \in \partial^{u}_D \mathcal{H}_j (v, w)$ for all $j \in J$, $\theta_s \in \partial^{u}_D \phi_s (v, w)$ for all $s \in S$, $w^{*} \in  \partial^{u}_D \Psi (v, w)$, $z^* \in N_D (0_{n_1 + n_2})$ such that 
 \begin{align}
  & \sum_{k=1}^{n} \xi^*_k  \big (\vartheta_k^{*'} + \Phi_k (v, w) \vartheta_k^{*''} \big) = - \sum_{j=1}^{p} \tau_j^* \mu_j - \sum_{s=1}^{q} \rho^*_s \ \theta_s - \eta^* w^{*} - z^* \label{eq17} \\
  & \tau_j^* \mathcal{H}_j (v, w) \geq 0, ~ \rho^*_s \phi_s (v, w) \geq 0, ~ \eta^* \Psi (v, w) \geq 0, ~ \forall ~ j \in J, ~ \forall ~ s \in S. \notag
 \end{align} 
	
As $(\bar{x}, \bar{y}) \in E,$ we have  
\begin{align*}
 & \tau_j^* \mathcal{H}_j (\bar{x}, \bar{y})	\leq \tau_j^* \mathcal{H}_j 
 (v, w), ~ \forall ~j \in J, \\
 & \, \rho^*_s \phi_s (\bar{x}, \bar{y}) \leq \rho^*_s \phi_s 
 (v, w),~ \forall ~ s \in S, \\
 & ~ \eta^* \Psi (\bar{x}, \bar{y}) \leq \eta^*\Psi (v, w).
\end{align*}
	
Since $\big(\mathcal{F}_k + \Phi_k (v, w) (-\mathcal{G}_k )\big)$ is $\partial^{u}_D$-pseudoconvex at $(v,w)$ for every $k \in \{1,\dots,n\}$, we have
 \begin{align}
  & \hspace{1.2cm} \langle (\vartheta_k^{*'} + \Phi_k (v, w) \vartheta_k^{*''}), (\bar{x}, \bar{y}) - (v, w) \rangle \  < 0 \notag \\ 
 & \Longrightarrow \, \sum_{k=1}^{n} \, \big\langle \xi^*_k (\vartheta_k^{*'} + \Phi_k (v, w) \vartheta_k^{*''}), (\bar{x}, \bar{y}) - (v, w) \big\rangle < 0.
  \label{eq18}
\end{align}
By using relation \eqref{eq17}, we get 
\begin{equation}\label{eq19}
 \Bigg\langle \sum_{j=1}^{p} \tau_j^* \mu_j + \sum_{s=1}^{q} \rho^*_s 
 \theta_s + \eta^* w^{*} + z^*, (\bar{x}, \bar{y}) - (v, w) \Bigg\rangle \  > 0.
\end{equation}
As $\mathcal{H}_j$ and $\phi_s$ are $\partial^{u}_D$-quasiconvex at $(v, w)$ for all $j \in J$ and all $s \in S$, and since $z^* \in N_D(0_{n_1 + n_2})$, we get 
\begin{align}
 & \big\langle \mu_j, (\bar{x}, \bar{y}) - (v, w) \big\rangle  \leq 0, ~~ \forall ~ j \in J_0 (\bar{x}, \bar{y}), \notag \\
 & \big\langle \theta_s, (\bar{x}, \bar{y}) - (v, w) \big\rangle \leq 0, ~~ \forall ~ s \in S_0 (\bar{x}, \bar{y}), \notag \\
 & \big\langle z^*, (\bar{x}, \bar{y}) - (v, w) \big\rangle \leq 0. \label{eq20}
\end{align}
Then 
\begin{align*}
 & \big\langle \tau_j^* \mu_j, (\bar{x}, \bar{y})- (v,w) \big\rangle  \leq 0, ~~ \forall ~ j \in J_0(\bar{x}, \bar{y}), \\
 & \big\langle \rho^*_s \theta_s, (\bar{x}, \bar{y}) - (v, w) \big\rangle \leq 0, ~~ \forall ~ s \in S_0 (\bar{x}, \bar{y}).
\end{align*}
If $j \in J_0^{\neq} (\bar{x}, \bar{y})$, then it follows that $\tau_j^* = 0$, while if $s \in S_0^{\neq} (\bar{x}, \bar{y})$, then $\rho^*_s=0$. Hence,
\begin{align}
 & \sum_{j=1}^{p} \tau_j^* \big\langle \mu_j, (\bar{x}, \bar{y}) - (v, w) \big\rangle 
 \leq 0, \label{eq21} \\
 & \sum_{s=1}^{q} \rho^*_s \big\langle \theta_s, (\bar{x}, \bar{y}) - (v, w) \big\rangle \leq 0, \label{eq22}
\end{align}
Now, we claim that
\begin{equation}\label{eq23}
 \big\langle \eta^* w^{*}, (\bar{x}, \bar{y})- (v, w) \big\rangle \leq 0.
\end{equation}
 Indeed, when $\eta^* > 0$, we have $\Psi(\bar{x}, \bar{y})\leq \Psi (a,b)$, 
 and since $\Psi$ is $\partial^{u}_D$-quasiconvex at $(v,w),$ we get $\big\langle \eta^* w^{*}, (\bar{x}, \bar{y}) - (v, w) \big\rangle \leq 0$, while when $\eta^* = 0$, we have $\big\langle \eta^* w^{*}, (\bar{x}, \bar{y}) - (v, w) \big\rangle = 0$. Hence relation \eqref{eq23} holds.

Therefore, it follows from \eqref{eq20}, \eqref{eq21}, \eqref{eq22} and 
\eqref{eq23} that 
\begin{equation}
 \Bigg\langle  \sum_{j=1}^{p} \tau_j^* \mu_j + \sum_{s=1}^{q} \rho^*_s \theta_s + \eta^* w^{*} + z^*,\, (\bar{x}, \bar{y}) - (v, w) \Bigg\rangle \leq 0.
\end{equation}
which contradicts \eqref{eq19}. Hence $\Phi(\bar{x}, \bar{y}) \nless \Phi(v,w)$,
and the proof is complete.
\end{proof}

Our final main result, which provides strong duality for the bilevel multiobjective 
optimization problem, is given below.

\begin{theorem}\textbf{(Strong  Duality):}\label{thm D2}
Let $(\check{x},\check{y})$ be a weak Pareto solution of \text{(A)} with ${\partial}^{u}_D$-nonsmooth ACQ holds. Then there exists $\delta^* = (\xi^*, \tau^*, \rho^*, \eta^*) \in \mathbb{R}_+^{(n+p+q+1)}$, with $\xi^* \neq 0_{\mathbb{R}^n}$, such that $(\check{x}, \check{y}, \delta^*)$ is a feasible point of \text{(DA)}. Fur\-ther\-mo\-re, if $(\mathcal{F}_k + \Phi_k (\check{x}, \check{y}) (-\mathcal{G}_k))$ is $\partial^{u}_D$-pseudoconvex at $(\check{x}, \check{y})$ for every $k \in \{1,\ldots,n\}$, $\mathcal{H}_j$ (with $j \in J_0 (\check{x}, \check{y})$), $\phi_s$ (with $s \in S_0 (\check{x}, \check{y})$) and $\Psi$ are $\partial^{u}_D$-quasiconvex at $(\check{x}, \check{y})$, then $(\check{x}, \check{y}, \delta^*)$ is a weak Pareto solution for problem
 $($\text{DA}$)$.
\end{theorem}

\begin{proof}
 Let $(\check{x}, \check{y})$ be a weak Pareto solution of \text{(A)} and such 
 that ${\partial}^{u}_D$-nonsmooth ACQ holds. Using Theorem \ref{thm N1}, 
 we can find $\xi^* \in (-\mathbb{R}^n_+)^\circ \backslash  \{0\}$ and $(\tau^*, \rho^*, \eta^*) \in  \mathbb{R}{^{(p+q+1)}_{+}}$ such that 
 \begin{align*}
  (0,0) \in & ~ \sum_{k=1}^{n}\xi^*_k  \big(\partial^{us}_D \mathcal{F}_k 
  (\check{x}, \check{y}) + \Phi_k (\check{x}, \check{y})\big(\partial^{us}_D
  (-\mathcal{G}_k)(\check{x},\check{y}) \big) + \sum_{j=1}^{p} \tau_j^* 
   \ \big(\partial^{u}_D \mathcal{H}_j (\check{x}, \check{y}) \\ 
  & + \sum_{s=1}^{q} \rho^*_s \ {\partial}^{u}_D \phi_s (\check{x}, 
  \check{y}) + \eta^*\ {\partial}^{u}_D \Psi (\check{x},\check{y}) + 
  N_D(0_{n_1+n_2}), \\
  & ~ \tau_j^* \mathcal{H}_j (\check{x}, \check{y}) = 0, ~ \forall ~ j \in J, \\ 
  & ~ \rho^*_s \phi_s (\check{x}, \check{y}) = 0, ~ \forall ~ s \in S.
\end{align*}
Since $\Psi (\check{x}, \check{y}) = 0$, $(\check{x}, \check{y},\delta^*)$ is 
a feasible point of \text{(DA)}.

Now, we will show that $(\check{x}, \check{y}, \delta^*)$ is a weak Pareto 
solution of $(\text{DA})$. Indeed, suppose for the contrary that there exists
$(\check{x_0}, \check{y_0}, \delta_0^*) \in \Omega$ such that
\begin{equation*}
 \Phi (\check{x}, \check{y}) - \Phi (\check{x_0}, \check{y_0}) \in -int \,
 (\mathbb{R}^n_+). 
\end{equation*}
Since $(\check{x}, \check{y}, \delta^*)$ and $(\check{x}, \check{y})$ are 
feasible points for \text{(A)} and \text{(DA)}, it follows from Theorem 
\ref{thm D1} that $\Phi (\check{x}, \check{y}) - \Phi (\check{x_0},
\check{y_0}) \notin -int\ \mathbb{R}^n_+$, a con\-tradiction.
\end{proof}

We finish the paper with the following example.

\begin{example}
Let us consider problem \text{$(Q_{1})$} again, addressed in Example \ref{ex2}, in order to analyze its Mond-Weir type dual problem
    $$(\text{MQ})\ \begin{cases}
        {\mathop{\mathbb{R}^{2}_+ - max }}  \left(\frac{\mathcal{F}_1 (v, w)}{\mathcal{G}_1 
 (v, w)},\, \frac{\mathcal{F}_2 (v, w)}{\mathcal{G}_2 (v, w)} \right) \\
s.~t. 
\\
 (0, 0) \in \sum_{k=1}^{2}\xi^*_k \Big({\partial}^{us}_D \mathcal{F}_k
 (v, w) + \Phi_k (v, w){\partial}^{us}_D (-\mathcal{G}_k) (v, w) \Big) + N_D 
 (0, 0) \\
 \vspace{0.1cm}
 \hspace{1.0cm} + \tau_1^* \ {\partial}^{u}_D \mathcal{H}_1 
 (v, w)+\tau_2^* \ {\partial}^{u}_D \mathcal{H}_2 
 (v, w) +  \rho^* \ {\partial}^{u}_D \phi_1 (v, w) 
 + \eta^*\ {\partial}^{u}_D \Psi (v, w) \\		
 \vspace{0.1cm}
 ~ \tau^*_1 \mathcal{H}_1 (v,w) \geq 0, ~\tau^*_2 \mathcal{H}_2 (v,w) \geq 0, \\  
 ~~ \rho^* \phi_s (v,w) \geq 0, \\ 
 \vspace{0.1cm}
 ~~~ \eta^* \Psi (v, w) \geq 0, \\
 \vspace{0.1cm}
 ~ \Upsilon^* = (\xi^*_{1}, \xi^*_2, \tau^*_1, \tau^*_2, 
 \rho^*, \eta^*) \geq 0, \\
 ~ (\xi^*_1, \xi^*_2) \neq 
 (0, 0),
 \end{cases}$$
where $\mathcal{F}_2 (v, w) = v + 2w + \frac{3}{2}$, $\mathcal{G}_2 (v, w) = 
v + \frac{1}{3}w + \frac{1}{2}$,
$$
\mathcal{F}_1(v, w) = \left\{
\begin{array}{ccl}
 v^2+2v+3 & {\rm if} & v \geq 0, \, w \geq 0, \\
 3v^2+v+4w^2-2w+4  & {\rm if} & v \geq 0, \, w < 0, \\
 w^2+2 & {\rm if} & v < 0,
\end{array}
\right.
$$
and
$$ 
\mathcal{G}_1(v, w) = \left\{
\begin{array}{ccl}
 v-2w+3 & {\rm if} & v \geq 0, \\
 1-v & {\rm if} & v < 0.
\end{array}
\right.
$$
Note that $\Phi_1(-1,0) = 1$ and $\Phi_2 (-1,0) = -1$. Using these values, we obtain $\varphi_2(v, w)= 2v+\frac{7}{3}w+2$ and 
$$
\varphi_1 (v, w) = \left\{
\begin{array}{ccl}
 v^2+v+2w & {\rm if} & v \geq 0, \, w \geq 0, \\
 3v^2+4w^2+1  & {\rm if} & v \geq 0, \, w < 0, \\
 v+w^2+1 & {\rm if} & v < 0,
\end{array}
\right.$$ where $\varphi_1(v, w)=\big(\mathcal{F}_1(v, w) + \Phi_1 (-1, 0) (-\mathcal{G}_1)(v, w)\big) $ and $\varphi_2(v, w)=\big(\mathcal{F}_2(v, w) + \Phi_2 (-1, 0) (-\mathcal{G}_2)(v, w)\big).$  

 On the one hand, $(-1, 0)$ is a feasible point of $(\text{MQ})$. Indeed, for  $\Upsilon^* =  \big(\frac{1}{2}, \frac{3}{2}, \frac{1}{4}, \frac{1}{2}, 1 ,\frac{1}{6}\big)$, where 
 $$\tau^*_1 \mathcal{H}_1 (-1, 0)=\frac{1}{2} \geq 0, ~\tau^*_2 \mathcal{H}_2 (-1, 0)=0 \geq 0,  $$
 $$\rho^* \phi_1 (-1, 0)=0 \geq 0, ~~~ \eta^* \Psi (-1, 0)=0 \geq 0.$$
 Since 
 $${\partial}^{us}_D \varphi_1 (-1, 0) = \big(1, 0\big),\,  {\partial}^{us}_D \varphi_2 (-1, 0) = \Big(2, \frac{7}{3}\Big),$$
 $${\partial}^{u}_D \mathcal{H}_1 
 (-1, 0)= (-1, -1),~{\partial}^{u}_D \mathcal{H}_2
 (-1, 0)=(0, 1), $$
 $$\ {\partial}^{u}_D \phi_1 (-1, 0)= (0, -1),~ {\partial}^{u}_D \Psi (-1, 0)=(0, 0),$$
we have 
 $$ \frac{1}{2}(1, 0)+\frac{3}{2}\Big(2, \frac{7}{3}\Big) +\frac{1}{4} (-1, -1)+\frac{1}{2} (0, 1)+  1(0, -1)+\frac{1}{6} (0, 0)+ \big(-\frac{13}{4}, -\frac{11}{4}\big)=(0, 0),$$ where $\big(-\frac{13}{4}, -\frac{11}{4}\big) \in N_D 
 (0, 0).$ 

We can check that $\varphi_1,\, \varphi_2$ are $\partial^{u}_D$-pseudoconvex at 
$(-1, 0)$ and $\mathcal{H}_1$, $\mathcal{H}_2$, $\phi$ and 
$\Psi$ are $\partial^{u}_D$-quasiconvex at $(-1, 0)$.

\begin{itemize}
    \item $ \big\langle {\partial}^{us}_D \varphi_1 (-1, 0),\, (v, w)-  (-1, 0)\big\rangle = v+1 \geq 0,$
$$ \Rightarrow \varphi_1 (v, w) = \left( \left\{
\begin{array}{ccl}
 v^2+v+2w & {\rm if} & v \geq 0, \, w \geq 0, \\
 3v^2+4w^2+1  & {\rm if} & v \geq 0, \, w < 0, \\
 v+w^2+1 & {\rm if} & v < 0,
\end{array}
\right. \right) \geq \varphi_1 (-1, 0). $$

\item $\big\langle {\partial}^{us}_D \varphi_2 (-1, 0),\, (v, w)-  (-1, 0)\big\rangle = 2v+\frac{7}{3}w+2 \geq 0,$
$$ \Rightarrow \varphi_2 (v, w) = 2v+\frac{7}{3}w+2 \geq \varphi_1 (-1, 0). $$

\item $\mathcal{H}_1(v, w) \leq  \mathcal{H}_1 (-1, 0) \Rightarrow 
- v - w \leq 1,$
$$\Rightarrow \big\langle {\partial}^{u}_D \mathcal{H}_1 (\check{x}, \check{y}),\, (v, w)-  (-1, 0)\big\rangle = -v-1-w \leq 0.$$

\item $\mathcal{H}_2(v, w) \leq  \mathcal{H}_2 (-1, 0) \Rightarrow 
 w \leq 0,$
$$\Rightarrow \big\langle {\partial}^{u}_D \mathcal{H}_2 (\check{x}, \check{y}),\, (v, w)-  (-1, 0)\big\rangle = w \leq 0.$$

\item $\phi(x,y) \leq \phi (\check{x}, \check{y}) \Rightarrow 
 w^2-w \leq 0,$
$$\Rightarrow \big\langle {\partial}^{u}_D \phi (-1, 0),\, (v, w)-  (-1, 0)\big\rangle = -w \leq 0.$$

\item $\Psi(v, w) \leq \Psi (-1, 0) \Rightarrow 
 \min\{w^2 - w^3, 0\} \leq 0,$
$$\Rightarrow \big\langle {\partial}^{u}_D \phi (-1, 0),\, (v, w)-  (-1, 0)\big\rangle \leq 0.$$
\end{itemize}

Then, for any feasible solution $(x, y) \in E = \mathbb{R}^+ \times \{0\}$ of problem $(\text{Q}_1)$ and any fea\-si\-ble solution $(v, w) \in \Omega$ of problem $(\text{MQ})$, we have 
$$\frac{\mathcal{F}_2 (x, y)}{\mathcal{G}_2 (x, y)}- \frac{\mathcal{F}_2 (v, w)}{\mathcal{G}_2 (v, w)} \geq 0,$$
i.e., Theorem \ref{thm D1} is applicable to both problems $(\text{Q}_1)$ and $(\text{MQ})$.
\end{example}

\section{Conclusions}   

We have discussed bilevel optimization problems with multiple fractional objectives, where the functions involved are not necessarily locally Lipschitz or continuous. We established necessary optimality conditions under a suitable constraint qualification based on directional upper convexificators. Furthermore, we derived sufficient optimality conditions under the assumptions of nonsmooth generalized convexity, namely $\partial^{u}_D$-pseudoconvexity and $\partial^{u}_D$-quasiconvexity. A Mond-Weir dual problem was also formulated, for which both weak and strong duality theorems were proven.

We hope our results offer new insights into nonsmooth multiobjective bilevel programs by utilizing techniques such as generalized subdifferentials for convexificators and generalized convex functions, as seen in works like \cite{KL,LK,ML2015}. The extension of these ideas to other problem classes will be the subject of subsequent work.

\bigskip

\noindent \textbf{Acknowledgements}
 This research was partially supported by ANID--Chile under project
 Fondecyt Regular 1241040 (Lara).

\end{document}